\documentclass[12pt]{amsart}
\usepackage[all]{xy}

\newtheorem{theorem}{Theorem}[section]

\newtheorem{proposition}[theorem]{Proposition}
\newtheorem{corollary}[theorem]{Corollary}
\newtheorem{problem}[theorem]{Problem}

\theoremstyle{definition}
\newtheorem{definition}[theorem]{Definition}
\newtheorem{example}[theorem]{Example}

\newtheorem{remark}[theorem]{Remark}
\newtheorem{algorithm}[theorem]{Algorithm}

\usepackage{amscd,amssymb}

\begin{document}

\title[Trees, algebras, invariants, and elliptic integrals]
{Planar trees, free nonassociative algebras, invariants,
and elliptic integrals}

\author[Vesselin Drensky and Ralf Holtkamp]
{Vesselin Drensky and Ralf Holtkamp}
\address{Institute of Mathematics and Informatics,
Bulgarian Academy of Sciences,
          1113 Sofia, Bulgaria}
\email{drensky@math.bas.bg}
\address{Fakult\"at f\"ur Mathematik, Ruhr-Universit\"at,
          44780 Bochum, Germany}
\email{ralf.holtkamp@ruhr-uni-bochum.de}

\thanks
{The work of the first author was partially supported by Grant
MI-1503/2005 of the Bulgarian National Science Fund.}


\begin{abstract}
We consider absolutely free algebras with (maybe infinitely) many multilinear
operations. Such multioperator algebras were introduced by Kurosh in 1960.
Multioperator algebras satisfy the Nielsen-Schreier property and
subalgebras of free algebras are also free. Free multioperator algebras
are described in terms of labeled reduced
planar rooted trees. This allows to apply combinatorial techniques to
study their Hilbert series and the asymptotics of their coefficients.
Then, over a field of characteristic 0,
we investigate the subalgebras of invariants under the action
of a linear group, their sets of free generators and their Hilbert series.
It has turned out that, except in the trivial cases, the algebra of invariants
is never finitely generated. In important partial cases the Hilbert series
of the algebras of invariants and the generating functions of
their sets of free generators are expressed in terms of elliptic integrals.
\end{abstract}

\maketitle

\section*{Introduction}

Let $K$ be an arbitrary field of any characteristic. Although
probably most of the $K$-algebras considered in the literature are
$K$-algebras equipped with just one binary operation, some
classical objects are equipped with more than one binary
operations, or even some non-binary operations. For example, quite
often Jordan algebras are considered with the usual multiplication
$u\circ v$ and one more ternary operation, the triple product
$\{uvw\}$. Poisson algebras have two binary operations -- the
Poisson bracket and the commutative and associative
multiplication. Apart from the classical examples, there are of
course also more recently introduced very important algebra types
that extend this list, such as dendriform dialgebras and
trialgebras. Then there are algebras with infinitely many
operations, such as homotopy algebras. The study of primitive
elements in free objects leads quite naturally to algebras with
infinitely many operations, cf.\ \cite{Lo, HLR}.

 In contrast to the case of the free associative algebra,
where the primitive elements form the free Lie algebra, Akivis
algebras (with up to ternary operations) were introduced in 1976
and later used to model the primitives of non-associative
algebras. Then Shestakov and Umirbaev \cite{SU} showed that there
exist primitive elements in the universal enveloping algebras of
free Akivis algebras which are not Akivis elements and gave a
description of the primitive elements. They arrived at algebras
with infinitely many operations, which they called hyperalgebras.

In 1960 Kurosh \cite{K2} introduced multioperator algebras (or
$\Omega$-algebras) as a generalization of multioperator groups
introduced by Higgins \cite{Hi} in 1956. In \cite{K2} Kurosh
established that free $\Omega$-algebras enjoy many of the
combinatorial properties of free nonassociative algebras (with one
binary operation) and suggested their simultaneous study.

Recall that a variety of universal algebras $\mathfrak M$ and its free
algebras satisfy the Schreier property if any subalgebra of a free
algebra of $\mathfrak M$ is free in $\mathfrak M$ again. For
example, free groups are Schreier. A result of Kurosh
\cite{K1} from 1947 states that absolutely free binary algebras are also
Schreier. In \cite{K2} Kurosh proved that free multioperator algebras
are Schreier again.
The variety $\mathfrak M$ satisfies the Nielsen property
if for any system of generators of a free subalgebra $S$ of a free
algebra of $\mathfrak M$ there exists an effective procedure (a
sequence of elementary transformations similar to the Nielsen
transformations in free groups) for obtaining a free set of
generators of $S$. In many important cases a variety satisfies the
Schreier property if and only if it satisfies the Nielsen
property, see Lewin \cite{L}. We shall say that such varieties and
their free algebras satisfy the Nielsen-Schreier property. See for
example the book by Mikhalev, Shpilrain and Yu \cite{MSY} for
different aspects of Nielsen-Schreier varieties. Schreier
varieties of multioperator algebras have been considered by
Burgin and Artamonov in \cite{BA}. In particular, they showed that
the Nielsen and Schreier properties are equivalent
for varieties of multioperator algebras defined by homogeneous polynomial
identities. A survey on the results before 1969 is given by Kurosh
in \cite{K3}. This article is also introductory for several other papers
published in the same issue of Uspehi Mat. Nauk (Russ. Math. Surv.)
and devoted to different aspects of free and close to free $\Omega$-algebras.

All above mentioned algebras are algebras defined over operads. In this
paper, we consider $\Omega$-algebras in the sense of Kurosh. Here $\Omega$
simply is a set of multilinear operations, which can be quite
arbitrary. The only restriction is
that we assume that $\Omega=\Omega_2\cup\Omega_3\cup\cdots$ is a
union of finite sets of $n$-ary operations $\Omega_n$, $n\geq 2$,
otherweise our quantitative results have no sense.
Then we consider the absolutely free nonunitary $\Omega$-algebra
$K\{X\}_{\Omega}$ freely generated by a set $X$. The
``$\Omega$-monomials'' form the free $\Omega$-magma
$\{X\}_{\Omega}={\mathcal Mag}_{\Omega}(X)$ which is a basis of
the vector space $K\{X\}_{\Omega}$ and can be described in terms
of labeled reduced planar rooted trees. In particular, if
$\Omega=\Omega_2$ consists of a single binary operation, then
$K\{X\}_{\Omega}=K\{X\}$ is the free nonassociative algebra and
$\{X\}={\mathcal Mag}_{\Omega}(X)={\mathcal Mag}(X)$ is the usual
free magma (the set of monomials in noncommuting nonassociative
variables). If $X=\{x\}$ consists of one element, then ${\mathcal
Mag}(X)$ is canonically identified with the set of planar rooted
binary trees. Another special case is when $\Omega_n$ consists of
one operation for each $n\geq 2$. Then we obtain the algebra
$K\{X\}_{\omega}$ and for $X=\{x\}$ we may identify
$\{X\}_{\omega}={\mathcal Mag}_{\omega}(X)$ with the set of all
reduced planar rooted trees.

Labeled reduced planar rooted trees have interesting combinatorics. This
allows to apply classical enumeration techniques from graph theory
and to study the Hilbert series of $K\{X\}_{\Omega}$ and the asymptotics of their coefficients.

Since free $\Omega$-algebras have bases which are easily
constructed in algorithmic terms, it is natural to develop a
theory of Gr\"obner (or Gr\"obner-Shirshov) bases. It was
surprising for us that the theory of Gr\"obner bases of free
$\Omega$-algebras is much simpler than the theory of Gr\"obner
bases of free associative algebras. For example, if an ideal of
$K\{X\}_{\Omega}$ is finitely generated then its Gr\"obner basis
is finite. If $J$ is a homogeneous ideal of $K\{X\}_{\Omega}$, we
express the Hilbert series of the factor algebra
$K\{X\}_{\Omega}/J$ in terms of the generating functions of
$\Omega$, $X$ and the Gr\"obner basis of $J$.

Further, we assume that the base field $K$ is of characteristic 0 and study subalgebras
of the invariants $K\{X\}_{\Omega}^G$ under the action
of a linear group $G$ on the free $\Omega$-algebra $K\{X\}_{\Omega}$ for a finite set
of free generators $X$, in the spirit of classical algebraic invariant theory and its generalization
to free and relatively free associative algebras, see the surveys \cite{Dr2, F1, KS}.
We show that the algebra of invariants $K\{X\}_{\Omega}^G$ is never finitely generated, except in
the obvious cases, when all invariants (if any) are expressed by $G$-invariant free generators.
The proof uses ideas of a similar result for relatively free Lie algebras,
see Bryant \cite{Br} and Drensky \cite{Dr1}.
Results of Formanek \cite{F1} and Almkvist, Dicks and Formanek \cite{ADF},
allow to express the Hilbert series of the algebra $K\{X\}_{\Omega}^G$ and the generating function
of the set of its free generators in terms of the Hilbert series of $K\{X\}_{\Omega}$,
which is an analogue of the Molien and Molien-Weyl formulas in commutative invariant theory.
In the important partial cases of a unipotent action of the infinite cyclic group $G$
and an action of the special linear group $SL_2(K)$
we give explicit expressions for the Hilbert series of the algebras
of invariants and the generating functions of their sets of free generators.
Applying these formulas to the free binary algebra $K\{X\}$ we express
the results in terms of elliptic integrals. We give similar formulas
when $\Omega$ has exactly one $n$-ary operation for each $n\geq 2$.

\section{Preliminaries}

We fix a field $K$ of any characteristic and a set of variables
$X$. In most of the considerations we assume that
the set $X$ is finite and $X=\{x_1,\ldots,x_d\}$. One of the main objects
in our paper is the absolutely free nonassociative and
noncommutative $K$-algebra $K\{X\}$ freely generated by the set $X$.
As a vector space it has a basis consisting of all non-associative words in the alphabet $X$.
For example, we make a difference between $(x_1x_2)x_3$ and
$x_1(x_2x_3)$ and even between $(xx)x$ and $x(xx)$. Words of
length $n$ correspond to planar binary trees with $n$ labeled leaves, see
e.g. \cite{GH, Ha}. We omit the parentheses when the
products are left normed. For example, $uvw=(uv)w$ and
$x^n=(x^{n-1})x$.

More generally, following Kurosh \cite{K2}, we consider a set $\Omega$ of multilinear
operations. We assume that $\Omega$ contains $n$-ary operations for $n\geq 2$ only
and for each $n$ the number of $n$-ary operations is finite.
We fix the notation
\[
\Omega_n=\{\nu_{ni}\mid i=1,\ldots,p_n\},\quad n=2,3,\ldots,
\]
for the set of $n$-ary operations. The free $\Omega$-magma
$\{X\}_{\Omega}={\mathcal Mag}_{\Omega}(X)$ consists of all
``$\Omega$-monomials'' and is obtained by recursion, starting with
$\{X\}_{\Omega}:=X$ and then continuing the process by
\[
\{X\}_{\Omega}:=\{X\}_{\Omega}\cup\{\nu_{ni}(u_1,\ldots,u_n)\mid
\nu_{ni}\in\Omega, u_1,\ldots,u_n\in \{X\}_{\Omega}\}.
\]
The set $\{X\}_{\Omega}$ is a $K$-basis of the free $\Omega$-algebra $K\{X\}_{\Omega}$,
and the operations of $K\{X\}_{\Omega}$ are defined using the multilinearity of
the operations in $\Omega$. We call the elements of $K\{X\}_{\Omega}$ $\Omega$-polynomials.

The free $\Omega$-magma $\{X\}_{\Omega}$
can be described also in terms of labeled reduced planar rooted trees.
Recall that a finite connected graph $\emptyset\not= T=(\text{Ve}(T),\text{Ed}(T))$,
with a distinguished vertex $\rho_T$, is called a rooted tree with root $\rho_T$,
if for every vertex $\lambda\in\text{Ve}(T)$ there is exactly one path
connecting $\lambda$ and $\rho_T$. Thinking of the edges as oriented towards the root,
at each vertex there are incoming edges and (except for the root) one outgoing edge.
The leaves of $T$ have no incoming edges and the root has no outgoing edges.
The tree is reduced if there are no edges with one incoming edge. A rooted tree $T$ with a chosen
order of incoming edges at each vertex is called a planar rooted tree.
We label the vertices $\lambda$ of the reduced planar rooted tree $T$ in the following way.
If $\lambda$ is not a leaf and has $n$ incoming edges, then we label it with an $n$-ary operation $\nu_{ni}$.
We call such trees $\Omega$-trees. If $\lambda$ is a leaf of the $\Omega$-tree $T$,
we label it with a variable $x_j\in X$. We refer to such trees as $\Omega$-trees with labeled leaves.
There is a one-to-one correspondence between the $\Omega$-monomials and the $\Omega$-trees with labeled leaves.
For example,
the monomial
\[
\nu_{31}(\nu_{23}(x_1,x_1),x_3,\nu_{32}(x_2,x_1,x_4))
\]
corresponds to the following tree:
\[
 \xymatrix{ {\bullet}_{x_1}\ar@{-}[dr]& {\bullet}_{x_1}\ar@{-}[d]& {\bullet}_{x_3}\ar@{-}[dd]
 & {\bullet}_{x_2}\ar@{-}[d]& {\bullet}_{x_1}\ar@{-}[dl]& {\bullet}_{x_4}\ar@{-}[dll]\\
& {\bullet}_{\nu_{23}}\ar@{-}[dr] & &{\bullet}_{\nu_{32}}\ar@{-}[dl] & & \\
& & {\bullet}_{\nu_{31}}
\\}
\]
\begin{center}Fig.\ 1
\end{center}

 We consider nonunitary algebras only. If we want to
deal with unitary algebras, we need certain coherence conditions,
because we have to express the monomials of the form
$\nu_{ni}(u_1,\ldots,1,\ldots,u_n)$, $u_j\in\{X\}_{\Omega}$, as
linear combinations of elements of $\{X\}_{\Omega}$, see e.g.
\cite{H}.

The algebra $K\{X\}_{\Omega}$ has a natural grading, defined by $\text{deg}(x_j)=1$, $x_j\in X$,
and then extended on the $\Omega$-monomials inductively by
\[
\text{deg}(\nu_{ni}(u_1,\ldots,u_n))=\sum_{j=1}^n\text{deg}(u_j),\quad u_j\in\{X\}_{\Omega}.
\]
Similarly, if $\vert X\vert=d$, then $K\{X\}_{\Omega}$ has a ${\mathbb Z}^d$-grading, or a multigrading,
counting the degree $\text{deg}_j(u)$ of any $\Omega$-word
$u$ in each free generator $x_j\in X$.
For a graded vector subspace $V$ of $K\{X\}_{\Omega}$ we consider the homogeneous component
$V^{(k)}$ of degree $k$. In the multigraded case the (multi)homogeneous component of $V$
of degree $(k_1,\ldots,k_d)$ is denoted by $V^{(k_1,\ldots,k_d)}$. The formal power series
with nonnegative integer coefficients
\[
H(V,t)=\sum_{k\geq 1}\dim(V^{(k)})t^k,
\]
\[
H(V,t_1,\ldots,t_d)=\sum_{k_j\geq 0}\dim(V^{(k_1,\ldots,k_d)})t_1^{k_1}\cdots t_d^{k_d},
\]
are called the Hilbert series of $V$ in the graded and multigraded cases, respectively.
Similarly, if $W$ is a set of (multi)homogeneous elements in $K\{X\}_{\Omega}$, the generating
function of $W$ is
\[
G(W,t)=\sum_{k\geq 1}\#(W^{(k)})t^k,
\]
\[
G(W,t_1,\ldots,t_d)=\sum_{k_j\geq 0}\#(W^{(k_1,\ldots,k_d)})t_1^{k_1}\cdots t_d^{k_d},
\]
where $\#(W^{(k)})$ and $\#(W^{(k_1,\ldots,k_d)})$ are the numbers of homogeneous elements
of the corresponding degree.

The above (multi)gradings work if the set $X$ of free generators
consists of $d$ elements. We may consider more general situation
of ${\mathbb Z}^d$-grading, when the set $X$ is arbitrary (but
still countable). We assign to each $x_j\in X$ a degree
\[
\text{deg}(x_j)=(a_{j1},\ldots,a_{jd}),\quad a_{jk}\geq 0,
\]
and assume that for each $(a_1,\ldots,a_d)\in {\mathbb Z}^d$
there is a finite number of generators of this degree. Again, the algebra
$K\{X\}_{\Omega}$ is ${\mathbb Z}^d$-graded and we may speak about the Hilbert
series of its graded vector subspaces and generating functions of its
subsets consisting of homogeneous elements.

\section{Hilbert series and their asymptotics}

The next result is standard and relates the Hilbert series of $K\{X\}_{\Omega}$
and the generating functions of its operations and generators.
Compare with the cases $K\{X\}$ and $K\{X\}_{\omega}$.
(For the relation between the Hilbert series of $K\{X\}$ with
any ${\mathbb Z}^d$-grading and $K\{x\}$ see e.g. Gerritzen \cite{G}
and Rajaee \cite{R}. For general references on enumeration techniques for graphs see
the book by Harary and Palmer \cite{HP}.)

\begin{proposition}\label{Hilbert series of free algebra}
Let
\[
G(\Omega,t)=\sum_{n\geq 2}\#(\Omega_n)t^n=\sum_{n\geq 2}p_nt^n
\]
be the generating function of the set $\Omega$.

{\rm (i)} The Hilbert series
\[
H(K\{x\}_{\Omega},t)=\sum_{k\geq 1}\dim(K\{x\}_{\Omega}^{(k)})t^k
\]
satisfies the functional equation
\begin{equation}\label{functional equation}
G(\Omega,H(K\{x\}_{\Omega},t))-H(K\{x\}_{\Omega},t)+t=0.
\end{equation}
The equation {\rm (\ref{functional equation})} and the condition
$H(K\{x\}_{\Omega},0)=0$ determine the Hilbert series
$H(K\{x\}_{\Omega},t)$ uniquely.

{\rm (ii)} If the set $X$ is ${\mathbb Z}^d$-graded in an arbitrary way, then
the Hilbert series of $K\{X\}_{\Omega}$ is
\[
H(K\{X\}_{\Omega},t_1,\ldots,t_d)=H(K\{x\}_{\Omega},G(X,t_1,\ldots,t_d)).
\]
\end{proposition}

\begin{proof}
(i) The Hilbert series of $K\{x\}_{\Omega}$ coincides with the generating function of
$\{x\}_{\Omega}$. The set of all elements $\nu_{ni}(u_1,\ldots,u_n)\not=x$
from $\{x\}_{\Omega}$, $u_j\in\{x\}_{\Omega}$, is in one-to-one correspondence with the set
\[
\{(\nu_{ni},u_1,\ldots,u_n)\mid \nu_{nj}\in\Omega,
u_j\in\{x\}_{\Omega}\}.
\]
For example, if we apply this correspondence to the
$\Omega$-monomial ($\Omega$-tree, respectively) given by the monomial
$\nu_{31}(\nu_{23}(x,x),x,\nu_{32}(x,x,x))$, then
$n=3, \nu_{ni}=\nu_{31}$, and $u_1, u_2, u_3$ are given by the
following $\Omega$-trees:

$\xymatrix{\\{\bullet}_{x}\ar@{-}[dr]& &
{\bullet}_{x}\ar@{-}[dl]
 & &{\bullet}_{x}\ar@{-}[dr]& {\bullet}_{x}\ar@{-}[d]& {\bullet}_{x}\ar@{-}[dl]\\
u_1 =& {\bullet}_{\nu_{23}} & , & u_2={\bullet}_{x},  & u_3 = & {\bullet}_{\nu_{32}}\\}$
\begin{center}Fig.\ 2
\end{center}

\bigskip

\noindent Hence
\[
H(K\{x\}_{\Omega},t)-t=G(\{x\}_{\Omega},t)-t=\sum_{k\geq 2}\#(\{x\}_{\Omega}^{(k)})t^k
\]
\[
=\sum_{n\geq 2}\#(\Omega_n)G(\{x\}_{\Omega},t)^n=G(\Omega,G(\{x\}_{\Omega},t))
=G(\Omega,H(K\{x\}_{\Omega},t)).
\]
The condition $H(K\{x\}_{\Omega},0)=0$ means that the formal power series has no constant term,
i.e., the algebra $K\{x\}_{\Omega}$ is nonunitary. Since $\Omega$ does not contain unary operations,
its generating function does not have constant and linear terms. This easily implies that the $k$-th coefficient
$\text{dim}(K\{x\}_{\Omega}^{(k)})$ of $H(K\{x\}_{\Omega},t)$
is determined by the first $k-1$ coefficients $\text{dim}(K\{x\}_{\Omega}^{(k)})$, $m=1,\ldots,k-1$, and
the Hilbert series $H(K\{x\}_{\Omega},t)$ is determined in a unique way.

(ii) Let us fix an $\Omega$-tree $T$ with $k$ leaves, and consider the set
of all possible ways to label the leaves of $T$ with elements of $X$. Clearly, the generating
function of this set (with respect to the ${\mathbb Z}^d$-grading on $X$)
is $G(X,t_1,\ldots,t_d)^k$. Hence
\[
H(K\{X\}_{\Omega},t_1,\ldots,t_d)=\sum_{k\geq 1}\#(\Omega\text{-trees with $k$ leaves})G(X,t_1,\ldots,t_d)^k
\]
\[
=H(K\{x\}_{\Omega},G(X,t_1,\ldots,t_d)).
\]
\end{proof}

\begin{remark}\label{formal series as these of free algebras}
If $G(\Omega,t)$ is the generating function of the set $\Omega$,
and $u(t_1,\ldots,t_d),v(t_1,\ldots,t_d)\in {\mathbb C}[[t_1,\ldots,t_d]]$
are formal power series such that
\begin{equation}\label{equation for u and v}
G(\Omega,v(t_1,\ldots,t_d))-v(t_1,\ldots,t_d)+u(t_1,\ldots,t_d)=0
\end{equation}
and satisfying
\begin{equation}\label{condition for u and v}
u(0,\ldots,0)=v(0,\ldots,0)=0
\end{equation}
then Proposition \ref{Hilbert series of free algebra} (i) gives that
\[
v(t_1,\ldots,t_d)=H(K\{x\}_{\Omega},u(t_1,\ldots,t_d)).
\]
Hence the functional equation (\ref{equation for u and v})
and the condition (\ref{condition for u and v})
determine uniquely the series $v(t_1,\ldots,t_d)$ as a function of $u(t_1,\ldots,t_d)$.
\end{remark}

\begin{example}\label{examples of Hilbert series}
(i) If $\Omega$ consists of one binary operation only, i.e.,
$K\{x\}_{\Omega}=K\{x\}$, then
$G(\Omega,t)=t^2$ and Proposition \ref{Hilbert series of free algebra} (i)
gives
\[
H(K\{x\},t)^2-H(K\{x\},t)+t=0.
\]
This equation has two solutions
\[
H(K\{x\},t)=\frac{1\pm\sqrt{1-4t}}{2}
\]
and the condition $H(K\{x\},0)=0$ implies that we have to choose the negative
sign. Hence
\begin{equation}\label{generating function for Catalan numbers}
H(K\{x\},t)=\frac{1-\sqrt{1-4t}}{2}
\end{equation}
is the well known generating function of the Catalan numbers.

(ii) If $\Omega_n$ consists of one operation $\nu_n:=\nu_{n1}$ for
each $n\geq 2$, i.e., $K\{x\}_{\Omega}=K\{x\}_{\omega}$, then
\[
G(\Omega,t)=t^2+t^3+\cdots=\frac{t^2}{1-t}.
\]
Hence $H(K\{x\}_{\omega},t)$ satisfies the equation
\[
\frac{H(K\{x\}_{\omega},t)^2}{1-H(K\{x\}_{\omega},t)}-H(K\{x\}_{\omega},t)+t=0,
\]
\[
2H(K\{x\}_{\omega},t)^2-(1+t)H(K\{x\}_{\omega},t)+t=0
\]
and the solution satisfying the condition $H(K\{x\}_{\omega},0)=0$ is
\begin{equation}\label{equation for omega}
H(K\{x\}_{\omega},t)=\frac{1+t-\sqrt{1-6t+t^2}}{4}.
\end{equation}
This is the generating function of the super-Catalan numbers (cf.
\cite{Sl} A001003).

 (iii) Let $\Omega=\underline{n}:=\{\nu_n\}$ consist of one
$n$-ary operation only, i.e., $G(\Omega,t)=G(\underline{n},t)=t^n$.
Then the Hilbert series of $K\{x\}_{\underline{n}}$ satisfies the
algebraic equation of degree $n$
\[
H(K\{x\}_{\underline{n}},t)^n-H(K\{x\}_{\underline{n}},t)+t=0
\]
and is equal to the generating function of the planar rooted $n$-ary trees.
\end{example}

\begin{remark}\label{Lagrange inversion formula}
If $f(z)$ is an analytic function in a neighbourhood of 0,
$f(0)\not=0$, and
\[
t=zf(z),
\]
then the Lagrange inversion formula gives that
\[
z=\sum_{k\geq 1}a_kt^k,\quad
a_k=\frac{1}{k!}\frac{d^{k-1}}{d\zeta^{k-1}}\left.\left(\frac{1}{f(\zeta)}\right)^k\right\vert_{\zeta=0}.
\]
The same holds if $f(z)$ is a formal power series with complex
coefficients and $f(0)\not=0$. Hence we may apply the formula for
$zf(z)=G(\Omega,z)$ and express $H(K\{x\}_{\Omega},t)$ in terms of
$G(\Omega,t)$.
\end{remark}

\begin{example}\label{applying Lagrange formula}

To obtain the coefficients of the Hilbert series
$H(K\{x\}_{\underline{n}},t)$ of Example \ref{examples of Hilbert
series} (iii) we apply Remark \ref{Lagrange inversion formula}.
(For an approach using Koszul duals of operads, compare also
\cite{BH}). We obtain $f(z)=1-z^{n-1}$,
\[
\left(\frac{1}{f(\zeta)}\right)^k=\frac{1}{(1-\zeta^{n-1})^k}
\]
\[
=1+\binom{k}{1}\zeta^{n-1}+\binom{k+1}{2}\zeta^{2(n-1)}+\binom{k+2}{3}\zeta^{3(n-1)}+\cdots,
\]
\[
a_k=\frac{1}{k!}\left. \frac{d^{k-1}}{d\zeta^{k-1}}\frac{1}{(1-\zeta^{n-1})^k}\right\vert_{\zeta=0},
\quad k\geq 1.
\]
Direct calculations show that
\[
a_k=\frac{1}{m(n-1)+1}\binom{mn}{m},\quad \text{ for } k=m(n-1)+1, m\geq 0,
\]
and $a_k=0$ otherwise. Hence
\[
H(K\{x\}_{\underline{n}},t) =\sum_{m\geq 0}
\binom{mn}{m}\frac{t^{m(n-1)+1}}{m(n-1)+1}
\]
\[
=t+t^n+nt^{2n-1}+\frac{n(3n-1)}{2}t^{3n-2}+\cdots.
\]
For small $m$ this can be seen also directly by counting the
planar rooted $n$-ary trees with the corresponding number of
leaves.

The set of $n$-ary trees with $n$ leaves consists of exactly one
tree, the so-called $n$-corolla. The set of $n$-ary trees with
$2n-1$ leaves consists of $n$ elements. The set of $n$-ary trees
with $3n-2$ leaves consists of $\binom{n}{2}+n^2$ elements,
typical examples (for $n=3$) are depicted in Fig. 3.

$\xymatrix{& & & & &  &
& {\bullet}\ar@{-}[dr]& {\bullet}\ar@{-}[d] &{\bullet}\ar@{-}[dl]\\
{\bullet}\ar@{-}[drr]& {\bullet}\ar@{-}[dr] &{\bullet}\ar@{-}[d]&
&{\bullet}\ar@{-}[d]& {\bullet}\ar@{-}[dl] &{\bullet}\ar@{-}[dll]
& {\bullet}\ar@{-}[dr]& {\bullet}\ar@{-}[d] &{\bullet}\ar@{-}[dl]\\
& &{\bullet}\ar@{-}[dr] & {\bullet}\ar@{-}[d] &
{\bullet}\ar@{-}[dl]& &
&{\bullet}\ar@{-}[dr]& {\bullet}\ar@{-}[d] &{\bullet}\ar@{-}[dl]\\
& & & {\bullet} & & & & & {\bullet}\\}$
\begin{center}Fig.\ 3
\end{center}

\bigskip

\noindent For $n=2$ we obtain the explicit formula for the Catalan numbers
\[
c_k=\frac{1}{k}\binom{2k-2}{k-1},\quad k=1,2,\ldots\ .
\]
\end{example}

\begin{example}\label{applying Lagrange formula again}
For $\Omega=\omega$, as in
Example \ref{examples of Hilbert series} (ii),
Remark \ref{Lagrange inversion formula} gives
\[
2z^2-(1+t)z+t=0,\quad t=\frac{z(1-2z)}{1-z},
\quad f(z)=\frac{1-2z}{1-z}.
\]
\[
\frac{1}{f^k(\zeta)}=\left(\frac{1-\zeta}{1-2\zeta}\right)^k=
\left(1+\frac{\zeta}{1-2\zeta}\right)^k
\]
\[
=1+\binom{k}{1}\frac{\zeta}{1-2\zeta}
+\binom{k}{2}\frac{\zeta^2}{(1-2\zeta)^2}+\cdots
\]
\[
+\binom{k}{k-1}\frac{\zeta^{k-1}}{(1-2\zeta)^{k-1}}
+\binom{k}{k}\frac{\zeta^k}{(1-2\zeta)^k}
\]
\[
=1+\binom{k}{1}\zeta(1+2\zeta+2^2\zeta^2+2^3\zeta^3+\cdots)
\]
\[
+\binom{k}{2}\zeta^2\left(1+\binom{2}{1}2\zeta+\binom{3}{1}2^2\zeta^2+\binom{4}{1}2^3\zeta^3+\cdots\right)
\]
\[
+\binom{k}{3}\zeta^3\left(1+\binom{3}{2}2\zeta+\binom{4}{2}2^2\zeta^2+\binom{5}{2}2^3\zeta^3+\cdots\right)+\cdots
\]
\[
+\binom{k}{k-1}\zeta^{k-1}\left(1+\binom{k-1}{k-2}2\zeta+\binom{k}{k-2}2^2\zeta^2+\binom{k+1}{k-2}2^3\zeta^3+\cdots\right)
\]
\[
+\binom{k}{k}\zeta^k\left(1+\binom{k}{k-11}2\zeta+\binom{k+1}{k-1}2^2\zeta^2+\binom{k+2}{k-1}2^3\zeta^3+\cdots\right),
\]
\[
a_k=\frac{1}{k!}\left. \frac{d^{k-1}}{d\zeta^{k-1}}
\left(\frac{1}{f(\zeta)}\right)^k\right\vert_{\zeta=0}
=\frac{1}{k}\left(\binom{k}{1}\binom{k-2}{0}
+\binom{k}{2}\binom{k-2}{1}2\right.
\]
\[
\left. +\binom{k}{3}\binom{k-2}{2}2^2
+\cdots+\binom{k}{k-1}\binom{k-2}{k-2}2^{k-2} \right)
\]
\[
=\frac{1}{2k}\left(\sum_{j=1}^{k-1}\binom{k}{j}\binom{k-2}{j-1}2^j
\right),\quad k=1,2,\ldots\ .
\]
Hence $a_k$ is the constant term of the Laurent polynomial
\[
\frac{1}{k}\zeta\left(1+\frac{1}{\zeta}\right)^k(1+2\zeta)^{k-2}.
\]
\end{example}

One of the important characteristics of a formal power series
$a(t)=\sum_{k\geq 0}a_kt^k$ is its radius of convergency
\[
r(a(t))=\frac{1}{\limsup_{k\to\infty}\sqrt[k]{a_k}}.
\]
By analogy with the (multilinear) codimension sequence for associative PI-algebras,
see Giambruno and Zaicev \cite{GZ}, we introduce the exponent of free $\Omega$-algebras.

\begin{definition}\label{exponent of free algebra}
Let $\vert X\vert=d<\infty$ and let
\[
H(K\{X\}_{\Omega},t)=\sum_{k\geq 1}a_kt^k
\]
be the Hilbert series of the free $\Omega$-algebra $K\{X\}_{\Omega}$. We define
the exponent of $K\{X\}_{\Omega}$ by
\[
\exp(K\{X\}_{\Omega})=\limsup_{k\to\infty}\sqrt[k]{a_k}.
\]
\end{definition}

It is easy to see that
\[
\exp(K\{X\}_{\Omega})=d\cdot\exp(K\{x\}_{\Omega}),
\]
i.e., it is sufficient to know the exponent of one-generated free $\Omega$-algebras.

\begin{example}
Let $\Omega=\underline{n}:=\{\nu_n\}$ consist of one $n$-ary
operation only. Applying the Stirling formula
\[
k!=\sqrt{2\pi k}\frac{k^ke^{\vartheta(k)}}{e^k},\quad \vert\vartheta(k)\vert<\frac{1}{12k},
\]
to Example \ref{applying Lagrange formula}, we obtain
\[
\exp(K\{x\}_{\underline{n}})=\lim_{m\to\infty}\sqrt[m(n-1)+1]{a_{m(n-1)+1}}
\]
\[
=\lim_{m\to\infty}\sqrt[m(n-1)+1]{\frac{1}{m(n-1)+1}\binom{mn}{m}}
=\lim_{m\to\infty}\sqrt[m(n-1)+1]{\binom{mn}{m}}
\]
\[
=\lim_{m\to\infty}\sqrt[m(n-1)+1]{\frac{n^{nm}}{(n-1)^{(n-1)m}}}=\frac{n}{n-1}\sqrt[n-1]n.
\]
Hence
\[
\lim_{n\to\infty}\exp(K\{x\}_{\underline{n}})=1.
\]
\end{example}

\begin{example}\label{exponent in case omega}
For $\Omega=\omega$, as in Example \ref{examples of Hilbert
series} (ii), in order to find the coefficient $a_k$ of the
Hilbert series of $K\{x\}_{\omega}$, we may expand the function
(\ref{equation for omega}) as a power series. Let
\[
\tau_{1,2}=3\pm 2\sqrt{2}
\]
be the zeros of $1-6t+t^2$. The function
\[
g_i=\sqrt{1-\tau_it},\quad i=1,2,
\]
is analytic in the open disc $\vert t\vert<1/\tau_i$ and its
radius of convergence is $1/\tau_i$. Since
$\sqrt{1-6t+t^2}=g_1(t)g_2(t)$ and the radius of convergence of
the product of two analytic functions is not less than the radius
of convergence of each of the factors, we conclude that
$r(H(K\{x\}_{\omega},t))\geq 1/\tau_1=\tau_2$. More precisely,
$r(H(K\{x\}_{\omega},t))=\tau_2$ because the derivatives of
$H(K\{x\}_{\omega},t)$ have singularities for $t=\tau_2$. Hence
\[
\exp(K\{x\}_{\omega})=\frac{1}{r(H(K\{x\}_{\omega},t))}=\tau_1\approx 5.8284.
\]
\end{example}

\begin{problem}
How does the exponent of $K\{x\}_{\Omega}$ depend on the analytic
properties of the generating function of $\Omega$? For
$\vert\Omega\vert<\infty$, express $\exp(K\{x\}_{\Omega})$ in
terms of the coefficients and the zeros of the polynomial
$f(z)=(z-G(\Omega,z))/z$. What happens if the number of operations
of degree $n$ is bounded by the same constant $a>0$ for all $n$
(or by $a n^k$ or by $a k^n$ for a fixed positive integer $k$)?
\end{problem}

\section{Nielsen-Schreier property and Gr\"obner bases}

We assume that the free $\Omega$-magma $\{X\}_{\Omega}$ is equipped with an admissible ordering $\prec$.
This means that the set $(\{X\}_{\Omega},\prec)$ is well ordered and if $u\prec v$ in $\{X\}_{\Omega}$, then
\[
\nu_{ni}(w_1,\ldots,w_{j-1},u,w_{j+1},\ldots,w_n)\prec \nu_{ni}(w_1,\ldots,w_{j-1},v,w_{j+1},\ldots,w_n)
\]
for any $\nu_{ni}\in \Omega$ and $w_1,\ldots,w_{j-1},w_{j+1},\ldots,w_n\in \{X\}_{\Omega}$.
If
\[
f=\sum_{i=1}^m\alpha_iu_i\in K\{X\}_{\Omega},\quad 0\not=\alpha_i\in K, u_i\in\{X\}_{\Omega},
u_1\succ \cdots\succ u_m,
\]
then $\overline{f}=u_1$ is the leading term of $u$.

\begin{example}\label{example of admissible order}
If $X=\{x_1,x_2,\ldots\}$ is countable, we order it by $x_1\prec x_2\prec\cdots$.
If $u,v\in\{X\}_{\Omega}$ and $\text{deg}(u)<\text{deg}(v)$, we assume that $u\prec v$.
If $\text{deg}(u)=\text{deg}(v)>1$,
\[
u=\nu_{n_1i_1}(u_1,\ldots,u_{n_1}),\quad v=\nu_{n_2i_2}(v_1,\ldots,v_{n_2}),
\]
we fix $u\prec v$ if $n_1<n_2$, or $n_1=n_2$, $i_1<i_2$ or, if $n_1=n_2$, $i_1=i_2$,
and $(u_1,\ldots,u_{n_1})\prec (v_1,\ldots,v_{n_1})$ lexicographically
(i.e., $u_1=v_1,\ldots,u_{k-1}=v_{k-1}$, $u_k\prec v_k$ for some $k$).
\end{example}

By a result of Kurosh \cite{K2} every subalgebra of the free $\Omega$-algebra
$K\{X\}_{\Omega}$ is free. His proof provides an algorithm which easily produces
a system of free generators of the subalgebra. We present this algorithm and some of
its consequences for self-containess of our exposition from the point of view
of admissible orders.

\begin{algorithm}\label{Nielsen-Schreier property}
Let $S$ be a subalgebra of $K\{X\}_{\Omega}.$ Assuming that the
base field $K$ is constructive and starting with any system $U$ of
generators of the subalgebra $S$, we want
to find a system of free generators of $S$.

Given $f_1,\ldots,f_m\in U$ which are algebraically dependent
(i.e., the homomorphism $K\{x_1,\ldots,x_m\}_{\Omega}\to S$
defined by $x_j\to f_j$, $j=1,\ldots,m$, has a nontrivial kernel),
the procedure suggests in each step an elementary transformation
which decreases one of the generators with respect to the
admissible ordering.

We may assume here that the leading coefficient of each $f_j$ is
equal to 1, i.e., $f_j=\overline{f_j}+\cdots$, where $\cdots$
denotes a linear combination of lower $\Omega$-monomials.

In the following we sketch how to find one generator $f_j$ such
that the
 $\Omega$-monomial $\overline{f_j}$ belongs to the subalgebra
generated by the other $\overline{f_i}, i\neq j$. (Then clearly we
can replace in the generating set $f_j$ by lower elements.)

 Let
$h(f_1,\ldots,f_m)=0$ for some $0\not=h(x_1,\ldots,x_m)\in
K\{x_1,\ldots,x_m\}_{\Omega}$. For every $\Omega$-monomial
$h_i(x_1,\ldots,x_m)\in \{x_1,\ldots,x_m\}_{\Omega}$,
\[
\overline{h_i(f_1,\ldots,f_m)}=h_i(\overline{f_1},\ldots,\overline{f_m})
\]
(because $h_i(\overline{f_1},\ldots,\overline{f_m})\not=0$).
Hence, there exist two different $h_1,h_2\in
\{x_1,\ldots,x_m\}_{\Omega}$ such that

\begin{equation}\label{equality of monomials}
h_1(\overline{f_1},\ldots,\overline{f_m})=h_2(\overline{f_1},\ldots,\overline{f_m}).
\end{equation}

If $h_1=x_j$ then
$\overline{f_j}=h_2(\overline{f_1},\ldots,\overline{f_m})$. Since
$h_2\not=x_j$, comparing the degrees of $\overline{f_j}$ and
$h_2(\overline{f_1},\ldots,\overline{f_m})$, we conclude that
$h_2(x_1,\ldots,x_m)=h_2(x_1,\ldots,x_{j-1},x_{j+1},\ldots,x_m)$
does not depend on $x_j$. The $\Omega$-polynomials
\[
f_1,\ldots,f_{j-1},f_j^{\ast}=f_j-h_2(f_1,\ldots,f_{j-1},f_{j+1},\ldots,f_m),f_{j+1},\ldots,f_m
\]
generate the same algebra as $f_1,\ldots,f_m$. If $f_j^{\ast}=0$,
then we may remove $f_j$ from the system of generators of $S$.
Otherwise, $\overline{f_j^{\ast}}\prec \overline{f_j}$ and the new
set $\{f_1,\ldots,f_j^{\ast},\ldots,f_m\}$ is lower in the
lexicographic ordering than $\{f_1,\ldots,f_m\}$.

The case where $h_2=x_k$ is similar. In the remaining case, where
 neither $h_1(x_1,\ldots,x_m)$ nor $h_2(x_1,\ldots,x_m)$ is equal to a single variable
 $x_j$, may be treated recursively in view of equation (\ref{equality of monomials}).
\end{algorithm}

The algorithm immediately gives the following well-known fact, see \cite{K2, K3, BA}.

\begin{corollary}\label{generators of graded subalgebras}
Every graded subalgebra $A$ of the ${\mathbb Z}^d$-graded
$\Omega$-algebra $K\{X\}_{\Omega}$ has a homogeneous system of
free generators.
\end{corollary}

Let $X=\{x_1,\ldots,x_d\}$. By analogy with the case of algebras
with one binary operation, we call the automorphism $\varphi$ of
the free algebra $K\{X\}_{\Omega}$ tame if it belongs to the
subgroup of $\text{Aut}(K\{X\}_{\Omega})$ generated by the linear
and triangular automorphisms, defined, respectively, by
\[
\varphi(x_j)=\sum_{i=1}^d\alpha_{ij}x_i,\quad \alpha_{ij}\in K,
j=1,\ldots,d,
\]
where the matrix $(\alpha_{ij})$ is invertible, and
\[
\varphi(x_j)=\alpha_jx_j+f_j(x_{j+1},\ldots,x_d),\quad
j=1,\ldots,d,
\]
where $\alpha_j\in K^{\ast}=K\setminus\{0\}$ and the
$\Omega$-polynomials $f_j(x_{j+1},\ldots,x_d)$ do not depend on
the variables $x_1,\ldots,x_j$. The discussion of Algorithm
\ref{Nielsen-Schreier property} shows also that all automorphisms
of $K\{X\}_{\Omega}$ are tame, which is a result of Burgin and Artamonov \cite{BA}.
If the base field $K$ is constructive, it provides an algorithm which decomposes the given
automorphism into a product of linear and triangular
automorphisms.

\begin{corollary}[Burgin and Artamonov \cite{BA}]\label{tame autos}
If $\vert X\vert<\infty$, then every automorphism of $K\{X\}_{\Omega}$ is tame.
\end{corollary}


The following gives a relation between the Hilbert series of graded subalgebras and generating
functions of their systems of free generators.

\begin{corollary}\label{free generators of subalgebras}
If $A$ is a graded subalgebra of the ${\mathbb Z}^d$-graded $\Omega$-algebra $K\{X\}_{\Omega}$
with Hilbert series $H(A,t_1,\ldots,t_d)$, then the generating function of
every homogeneous free generating set $Y$ of $A$ is
\begin{equation}\label{generating function of free generators of subalgebras}
G(Y,t_1,\ldots,t_d)=H(A,t_1,\ldots,t_d)-G(\Omega,H(A,t_1,\ldots,t_d)).
\end{equation}
\end{corollary}

\begin{proof} It is sufficient to use Corollary \ref{generators of graded subalgebras}
and to apply (\ref{equation for u and v})
in Remark \ref{formal series as these of free algebras}.
\end{proof}

Existence of admissible orderings allows to develop the theory of
Gr\"obner (or Gr\"obner-Shirshov) bases of $\Omega$-ideals $J$ in
$K\{X\}_{\Omega}$. The obvious definition of $\Omega$-subwords of
an $\Omega$-word (or an $\Omega$-monomial)
$u=\nu_{ni}(u_1,\ldots,u_n)\in \{X\}_{\Omega}$ is by induction.
The subwords of $u$ are the words $u_j$ and the subwords of the
$u_j$. Now we fix an admissible ordering on $\{X\}_{\Omega}$. The
subset $B=B(J)$ of the ideal $J$ of $K\{X\}_{\Omega}$ is a
Gr\"obner basis of $J$ if, for every nonzero $f\in J$, there is an
element $g\in B$ such that the leading term $\overline{g}$ of $g$
is a subword of the leading term $\overline{f}$ of $f$. Most of
the standard properties of Gr\"obner bases for free noncommutative
algebras hold also for free $\Omega$-algebras. In particular, the
set of normal words, i.e., the $\Omega$-words which do not contain
as subwords $\overline{g}$, $g\in B(J)$, form a basis of the
factor algebra $K\{X\}_{\Omega}/J$. The set
\[
I=I(J)=\{\overline{f}\mid 0\not=f\in J\} \subseteq \{X\}_{\Omega}
\]
is an ideal of $\{X\}_{\Omega}$ generated by the set
$\{\overline{g}\mid g\in B(J)\}$. We call $I$ the initial ideal of
$J$.

There is an algorithm to compute the Gr\"obner basis of an ideal $J$
of the free associative algebras $K\langle X\rangle$ which is an analogue of the Buchberger
algorithm for ideals of polynomial algebras. (Compare with the approach based on the
diamond lemma in the paper by Bergman \cite{Be}. See also the survey article by Ufnarovski \cite{U}.)
If the ideal $J$ of $K\langle X\rangle$ is generated by a set $\{f_k\}$,
we fix an admissible ordering and start the construction of the Gr\"obner basis $B(J)$
defining $B(J):=\{f_k\}$. If the leading terms of $f_1,f_2\in B(J)$ are $\overline{f_1},\overline{f_2}$,
respectively, and $u_1\overline{f_1}v_1=u_2\overline{f_2}v_2$ for some monomials $u_k,v_k$, $k=1,2$, then,
for suitable nonzero $\alpha_1,\alpha_2\in K$, the $S$-polynomial
$f_{12}=\alpha_1u_1f_1v_1-\alpha_2u_2f_2v_2$, if not 0,  is lower
than $u_1f_1v_1$ and $u_2f_2v_2$ in the admissible ordering of
$K\langle X\rangle$. If $f_{12}\not=0$, we have an ``ambiguity'' and, in order to solve it,
we add $f_{12}$ to $B(J)$. We have two kinds of $S$-polynomials. In the first case,
$\overline{f_1}$ and $\overline{f_2}$ overlap, i.e., $u_1\overline{f_1}=\overline{f_2}v_2$.
In the second case $\overline{f_2}$ is a subword of $\overline{f_1}$, i.e.,
$\overline{f_1}=u_2\overline{f_2}v_2$ (or $\overline{f_1}$ is a subword of $\overline{f_2}$).
In the case of $\Omega$-algebras there are no overlaps and it is sufficient to
consider $S$-polynomials obtained when one of the leading terms is an $\Omega$-subword of the other.
Hence the Buchberger algorithm has the following form.
(Of course, we fix an admissible ordering of $K\{X\}_{\Omega}$ and
assume that the base field $K$ is constructive.)

\begin{algorithm}\label{Buchberger algorithm}
Let the $\Omega$-ideal $J$ of $K\{X\}_{\Omega}$ be generated by the set $\{f_k\}$.
We may assume that the coefficients of the leading terms $\overline{f_k}$ are all equal to 1.
We define $B(J):=\{f_k\}$. Let $\overline{f_2}$ be an $\Omega$-subword of $\overline{f_1}$
for some $f_1,f_2\in B(J)$, e.g. $\overline{f_1}=\nu_{ni}(u_1,\ldots,u_n)$, where
$u_1,\ldots,u_n\in\{X\}_{\Omega}$ and $\overline{f_2}$ is a subword of some $u_j$.
Then we replace in $B(J)$ the $\Omega$-polynomial $f_1$ by
\[
\widetilde{f_1}=f_1-\nu_{ni}(u_1,\ldots,\widetilde{u_j},\ldots,u_n),
\]
where the $\Omega$-polynomial $\widetilde{u_j}$ is obtained from
the $\Omega$-monomial $u_j$ by replacing $\overline{f_2}$ by
$f_2$. If $\widetilde{f_1}\not=0$, we norm it (making the leading
coefficient equal to 1). If $\widetilde{f_1}=0$, we remove it from
$B(J)$. We continue the process as long as possible.
\end{algorithm}


\begin{proposition}\label{Groebner basis of finitely generated ideal}
Finitely generated ideals of $K\{X\}_{\Omega}$ have finite Gr\"ob\-ner bases
with respect to any admissible ordering.
\end{proposition}

\begin{proof}
Let the ideal $J$ be generated by $f_1,\ldots,f_m$. Following
Algorithm \ref{Buchberger algorithm}, we start with
$B(J):=\{f_1,\ldots,f_m\}$. In each step we either remove one of the
elements of $B(J)$ or replace it with a new polynomial, without adding more
elements to $B(J)$. In a finite number of steps the procedure will
stop and we obtain a finite Gr\"obner basis of $J$.
\end{proof}

\begin{remark}
Proposition \ref{Groebner basis of finitely generated ideal}
implies the solvability of the word problem for $\Omega$-algebras.
This means that if $A$ is a finitely presented $\Omega$-algebra,
there is an algorithm which decides whether an element $f\in A$ is
equal to 0. In other words, if $A \cong
K\{x_1,\ldots,x_d\}_{\Omega}/J$ for a finitely generated $\Omega$-ideal
$J$, and the generators $f_1,\ldots,f_m$ of $J$ are explicitly
given, then we can decide whether $f\in
K\{x_1,\ldots,x_d\}_{\Omega}$ belongs to $J$. One should pay
attention to the fact that the solvability of decision problems
cannot be transferred to factor algebras. There exist finitely
generated ideals $J_0$ of the free associative algebra $K\langle
x_1,\ldots,x_d\rangle$, such that the word problem has no solution
in $A\cong K\langle x_1,\ldots,x_d\rangle/J_0$. Of course, in this
case $A\cong K\{x_1,\ldots,x_d\}/J$ for some ideal $J$ of
$K\{x_1,\ldots,x_d\}$ and $J$ is not finitely generated.
\end{remark}

We conclude this section with an $\Omega$-analogue of a theorem of Rajaee \cite{R} for $K\{X\}$.
We assume that $K\{X\}_{\Omega}$ is ${\mathbb Z}^d$-graded. Then, as in \cite{R}, the reduced Gr\"obner basis
of a (multi)homogeneous $\Omega$-ideal $J$ of $K\{X\}_{\Omega}$ consists of (multi)homogeneous elements.

\begin{theorem}\label{theorem of Rajaee}
Let $J$ be a homogeneous $\Omega$-ideal of $K\{X\}_{\Omega}$ with respect to any ${\mathbb Z}^d$-grading
of $K\{X\}_{\Omega}$. Then the Hilbert series of the factor algebra
$A\cong K\{X\}_{\Omega}/J$ and the generating functions of the set of generators $X$ and of
the reduced Gr\"obner basis $B(J)$ of $J$ with respect to any admissible ordering are related by
\begin{equation}\label{generating function of Groebner basis}
H(A,t_1,\ldots,t_d)=H(K\{x\}_{\Omega},G(X,t_1,\ldots,t_d)-G(B(J),t_1,\ldots,t_d)).
\end{equation}
\end{theorem}

\begin{proof}
We follow the idea of the proof in \cite{R}.
For convenience, we denote the Hilbert series $H(P,t_1,\ldots,t_d)$ or the generating
function $G(P,t_1,\ldots,t_d)$ of the graded object $P$ by
$H(P)$ and $G(P)$, respectively.
Clearly, the isomorphism $A\cong K\{X\}_{\Omega}/J$ implies
$H(A)=H(K\{X\}_{\Omega})-H(J)$. Also, $H(K\{X\}_{\Omega})=G(\{X\}_{\Omega})$
and $H(J)=G(I)$, where $I=I(J)\triangleleft \{X\}_{\Omega}$ is the initial
ideal of $J$. Finally, $G(B(J))=G(B(I))$, where $B(I)$ is the minimal
generating set of the $\Omega$-ideal $I$. Hence (\ref{generating function of Groebner basis}) is equivalent to
\[
G(\{X\}_{\Omega})-G(I)=H(K\{x\}_{\Omega},G(X)-G(B(I))).
\]
The elements of $I$ which do not belong to the minimal set of
generators $B(I)$ are characterized by the property that they are
of the form
\[
u=\nu_{ni}(v_1,\ldots,v_{j-1},w_j,v_{j+1},\ldots,v_n),\quad
v_k\in\{X\}_{\Omega},\quad w_j\in I.
\]
Hence
\[
I=\left(\bigcup_{\nu_{ni}\in\Omega}\bigcup_{j=1}^n
\nu_{ni}(\underbrace{\{X\}_{\Omega},\ldots,\{X\}_{\Omega}}_{j-1\text{ times}},I,
\underbrace{\{X\}_{\Omega},\ldots,\{X\}_{\Omega}}_{n-j\text{ times}})\right)\bigcup B(I),
\]
\[
G(I)=\sum_{\nu_{ni}\in\Omega}G\left(
\bigcup_{j=1}^n
\nu_{ni}(\underbrace{\{X\}_{\Omega},\ldots,\{X\}_{\Omega}}_{j-1\text{ times}},I,
\underbrace{\{X\}_{\Omega},\ldots,\{X\}_{\Omega}}_{n-j\text{ times}})
\right) +G(B(I)).
\]
By the principle of inclusion and exclusion,
\[
G\left(
\bigcup_{j=1}^n
\nu_{ni}(\underbrace{\{X\}_{\Omega},\ldots,\{X\}_{\Omega}}_{j-1\text{ times}},I,
\underbrace{\{X\}_{\Omega},\ldots,\{X\}_{\Omega}}_{n-j\text{ times}})
\right)
\]
\[
=\sum_{k=1}^n(-1)^{k-1}\binom{n}{k}G^{n-k}(\{X\}_{\Omega})G^k(I)
\]
\[
=G^n(\{X\}_{\Omega})-(G(\{X\}_{\Omega})-G(I))^n.
\]
This implies
\[
G(I)=\sum_{\nu_{ni}\in\Omega}\left(
G^n(\{X\}_{\Omega})-(G(\{X\}_{\Omega})-G(I))^n
\right) +G(B(I))
\]
\[
=G(\Omega,G(\{X\}_{\Omega}))-G(\Omega,G(\{X\}_{\Omega})-G(I))+G(B(I)).
\]
Applying (\ref{functional equation}) and Proposition \ref{Hilbert series of free algebra} (ii) we obtain
\[
G(I)=G(\{X\}_{\Omega})-G(X)-G(\Omega,G(\{X\}_{\Omega})-G(I))+G(B(I)),
\]
\[
G(\Omega,G(\{X\}_{\Omega})-G(I))-(G(\{X\}_{\Omega})-G(I))+(G(X)-G(B(I)))=0.
\]
By (\ref{equation for u and v}) in Remark \ref{formal series as
these of free algebras} we conclude that
\[
H(A)=G(\{X\}_{\Omega})-G(I)=H(K\{x\}_{\Omega},G(X)-G(B(I)))
\]
and this completes the proof.
\end{proof}

\begin{corollary}\label{explicit generating function of Groebner basis}
Let $J$ be a homogeneous $\Omega$-ideal of $K\{X\}_{\Omega}$ with respect to any ${\mathbb Z}^d$-grading
of $K\{X\}_{\Omega}$. Let $B(J)$ be the reduced Gr\"obner basis of $J$
with respect to any admissible ordering. Then
\[
G(B(J),t_1,\ldots,t_d)=G(\Omega,H(K\{X\}_{\Omega}/J,t_1,\ldots,t_d))
\]
\[
-H(K\{x\}_{\Omega}/J,t_1,\ldots,t_d)+G(X,t_1,\ldots,t_d).
\]
\end{corollary}

\begin{proof}
Applying (\ref{equation for u and v})
in Remark \ref{formal series as these of free algebras}
to (\ref{generating function of Groebner basis})
we obtain for $A\cong K\{X\}_{\Omega}/J$ that
\[
G(\Omega,H(A))-H(A)+G(X)-G(B(J))=0,
\]
which gives the expression for the generating function of
the reduced Gr\"obner basis $B(J)$ of $J$.
\end{proof}

\begin{example}
Let $\Omega=\underline{n}=\{\nu_n\}$ consist of one $n$-ary
operation only, as in Example \ref{examples of Hilbert series}
(iii), and let $A$ be the free commutative and associative $n$-ary
algebra in one variable, i.e., $A$ is the homomorphic image of
$K\{x\}_{\underline{n}}$ modulo the ideal generated by all
\[
\nu_n(u_1,\ldots,u_n)-\nu_n(u_{\sigma(1)},\ldots,u_{\sigma(n)}),
\quad \sigma\in S_n,
\]
\[
\nu_n(u_1,\ldots,\nu_n(u_j,v_2,\ldots,v_n),\ldots,u_n)
\]
\[
-
\nu_n(\nu_n(u_1,v_2,\ldots,v_n),\ldots,u_j,\ldots,u_n),\quad
j=2,\ldots,n,
\]
where $S_n$ is the symmetric group and $u_j,v_k\in
K\{x\}_{\underline{n}}$. Hence the homogeneous component $A^{(k)}$
is one-dimensional for $k=(n-1)m+1$ and is equal to zero for all
other $k$. We may assume that $A^{((n-1)m+1)}$ is spanned by
\[
x^{(n-1)m+1}=\nu_n(x^{(n-1)(m-1)+1},x,\ldots,x),\quad m\geq 1.
\]
Since $G(\Omega,t)=t^n$ and
\[
H(A,t)=\sum_{m\geq 0}t^{(n-1)m+1}=\frac{t}{1-t^{n-1}},
\]
Corollary \ref{explicit generating function of Groebner basis} gives that the generating function of
the reduced Gr\"obner basis with respect to any admissible ordering is
\[
G(B(J),t)=\left(\frac{t}{1-t^{n-1}}\right)^n-\frac{t}{1-t^{n-1}}+t
\]
\[
=t^n\sum_{m\geq 1}\left(\binom{m+n-1}{n-1}-1\right)t^{(n-1)m}.
\]
We fix the admissible ordering on $\{x\}_{\underline{n}}$ which
compares the monomials first by degree and then by inverse
lexicographic ordering: If $u=\nu_n(u_1,\ldots,u_n)$,
$v=\nu_n(v_1,\ldots,v_n)$, then $u\prec v$ if either
$\text{deg}(u)<\text{deg}(v)$ or $\text{deg}(u)=\text{deg}(v)$ and
$u_k\prec v_k$, $u_{k+1}=v_{k+1},\ldots,u_n=v_n$. In this way
$x^{(n-1)m+1}$ is the smallest monomial of degree $(n-1)m+1$. Then
the reduced Gr\"obner basis of $J$ consists of all
\[
\nu_n(x^{(n-1)m_1+1},\ldots,x^{(n-1)m_n+1})-x^{m(n-1)+1},
\]
$m_1+\cdots+m_n=m$, $(m_1,\ldots,m_n)\not=(m,0,\ldots,0)$.
The number of such polynomials of degree $(n-1)m+1$ is equal to the number
of all monomials $z_1^{m_1}\cdots z_n^{m_n}\not=z_1^m$ of total degree $m$
in $n$ variables $z_1,\ldots,z_n$.
\end{example}

\section{Invariant theory}

Till the end of the paper we fix a field $K$ of characteristic 0.
We assume that the set $X=\{x_1,\ldots,x_d\}$ is finite and
consists of $d$ elements. The general linear group $GL_d(K)$ acts
canonically on the $d$-dimensional vector space $KX$ with basis
$X$ and we identify it with the group of invertible
$d\times d$ matrices. If
\[
g=\left(\begin{matrix}
\alpha_{11}&\cdots&\alpha_{1d}\\
\vdots&\ddots&\vdots\\
\alpha_{d1}&\cdots&\alpha_{dd}\\
\end{matrix}\right)\in GL_d(K),\quad \alpha_{pq}\in K,
\]
then the action on $KX$ is defined by
\[
g(x_j)=\alpha_{1j}x_1+\cdots+\alpha_{dj}x_d,\quad j=1,\ldots,d.
\]
This action is extended diagonally on $K\{X\}_{\Omega}$ by
\[
g(u(x_1,\ldots,x_d))=u(g(x_1),\ldots,g(x_d)),\quad
u(x_1,\ldots,x_d)\in K\{X\}_{\Omega}.
\]
If $G$ is a subgroup of $GL_d(K)$, then the algebra of
$G$-invariants is defined in the obvious way as
\[
K\{X\}_{\Omega}^G=\{f\in K\{X\}_{\Omega}\mid g(f)=f,\quad g\in
G\}.
\]

One of the main problems in classical invariant theory is the
problem for finite generation of the algebra of invariants which
is a partial case of the 14th Hilbert problem. The same problem
has been intensively studied in noncommutative invariant theory.
See the surveys \cite{F1} and \cite{Dr2} for free and relatively
free associative algebras and the papers \cite{Br} and \cite{Dr1}
for free and relatively free Lie algebras. It has turned out that
in the noncommutative case the algebra of invariants is finitely
generated in very special cases only.

For example, if $G$ is a finite linear group acting on the free
associative algebra $K\langle X\rangle$, a theorem established independently
by Dicks and Formanek \cite{DF} and Kharchenko \cite{Kh2}
states that $K\langle X\rangle^G$ is finitely generated if and only if
$G$ is cyclic and acts on $KX$ by scalar multiplication. A simplified version of the proof
of Kharchenko is given by Dicks in \cite{C2}.
Koryukin \cite{Ko1} considered the case of the action of any linear group $G$ on $K\langle X\rangle$.
Let $d_1$ be the minimal integer with the property that there exist linearly independent $y_1,\ldots,y_{d_1}$
in the vector space $KX$ such that
$K\langle X\rangle^G\subset K\langle y_1,\ldots,y_{d_1}\rangle$. Changing linearly the system of
free generators $X=\{x_1,\ldots,x_d\}$, we may assume that
$K\langle X\rangle^G\subset K\langle x_1,\ldots,x_{d_1}\rangle$. The theorem of Koryukin \cite{Ko1} gives that
$K\langle X\rangle^G$ is finitely generated if and only if $G$ acts on
$Kx_1\oplus\cdots\oplus Kx_{d_1}$ as a finite cyclic group of scalar multiplications.

In the case of the free Lie algebra $L(X)$  Bryant \cite{Br} showed that $L(X)^G$
is not finitely generated if $G\not=\langle e\rangle$ is any finite group. The same result holds from \cite{Dr1}.

It is natural to expect that
something similar holds for free $\Omega$-algebras. If $J$ is an $\Omega$-ideal of
$K\{X\}_{\Omega}$ which is $GL_d(K)$-invariant, i.e.,
$GL_d(K)(J)=J$, then the action of $GL_d(K)$ on $K\{X\}_{\Omega}$
induces an action on the factor algebra $K\{X\}_{\Omega}/J$. Hence
the $G$-invariants of $K\{X\}_{\Omega}$ go to $G$-invariants of
$K\{X\}_{\Omega}/J$. If the group $G$ is finite or, more
generally, acts as a reductive group on $KX$, then every invariant
of $K\{X\}_{\Omega}/J$ can be lifted to an invariant of
$K\{X\}_{\Omega}$. Hence, in this case we may study $G$-invariants
of factor algebras and then lift the obtained results to the
algebra $K\{X\}_{\Omega}^G$ itself.

In the case of ordinary polynomial algebras, if the group $G$ is
finite, the algebra of invariants $K[X]^G$ is always nontrivial
and even of the same transcendence degree $d$ as $K[X]$. Lifting
the invariants, we obtain that the algebras $K\langle X\rangle^G$
and $K\{X\}^G$ are also nontrivial. In the case of (non-binary)
free $\Omega$-algebras, the picture is completely different:

\begin{example}\label{zero algebra of invariants}
Let $G=\{e,-e\}$, where $e$ is the identity $d\times d$ matrix and
let $\Omega=\underline{3}=\{\nu_3\}$ consist of a single ternary operation. Since
\[
(-e)(u(x_1,\ldots,u_d))=(-1)^ku(x_1,\ldots,u_d),
\]
\[
k=\text{deg}(u),u(x_1,\ldots,u_d)\in \{X\}_{\underline{3}},
\]
$K\{X\}_{\underline{3}}^G$ is spanned by all homogeneous monomials of
even degree. By Example \ref{applying Lagrange formula} (or by
easy induction), $K\{X\}_{\underline{3}}$ is spanned by monomials of
odd degree only. Hence $K\{X\}_{\underline{3}}^G=\{0\}$.
\end{example}

Let $T$ be an $\Omega$-tree with $N$ leaves. It is a reduced tree such that
every internal vertex (i.e., vertex which is not a leaf) is
labeled by an element of $\Omega_n$, where $n$ is the number of
incoming edges of the vertex. We denote by $\nu_T$ the
corresponding composition of operations from $\Omega$. If we label
the leaves of $T$ by $x_1,\ldots,x_N$ and denote the corresponding
$\Omega$-monomial by $\nu_T(x_1,\ldots,x_N)$, then the labeling of
the leaves of $T$ by $x_{j_1},\ldots,x_{j_N}$ gives rise to the
monomial $\nu_T(x_{j_1},\ldots,x_{j_N})$. For example, if $T$ is
the $\Omega$-tree in Fig. 4, then
\[
\nu_T(x_1,\ldots,x_6)=\nu_{31}(\nu_{23}(x_1,x_2),x_3,\nu_{32}(x_4,x_5,x_6)),
\]
and $\nu_T(x_1,x_1,x_3,x_2,x_1,x_4)$  corresponds to the $\Omega$-tree
with labeled leaves given in Fig. 1.

\[
 \xymatrix{ {\bullet}_{x_1}\ar@{-}[dr]& {\bullet}_{x_2}\ar@{-}[d]& {\bullet}_{x_3}\ar@{-}[dd]
 & {\bullet}_{x_4}\ar@{-}[d]& {\bullet}_{x_5}\ar@{-}[dl]& {\bullet}_{x_6}\ar@{-}[dll]\\
& {\bullet}_{\nu_{23}}\ar@{-}[dr] & &{\bullet}_{\nu_{32}}\ar@{-}[dl] & & \\
& & {\bullet}_{\nu_{31}}
\\}
\]
\begin{center}Fig.\ 4
\end{center}

\noindent Clearly, the $GL_d(K)$-module
$\nu_T(KX,\ldots,KX)$ is isomorphic to the $N$-th tensor power
$(KX)^{\otimes N}$ by the isomorphism which deletes the operations
\begin{equation}\label{module structure for fixed tree}
\pi_T:\nu_T(x_{j_1},\ldots,x_{j_N})\to x_{j_1}\otimes\cdots\otimes
x_{j_N}.
\end{equation}

As in the classical case, if $G$ is a subgroup of $GL_d(K)$, then
the algebra of invariants $K\{X\}_{\Omega}^G$ is graded with
respect to the usual grading defined by $\text{deg}(x_j)=1$,
$j=1,\ldots,d$. Even more holds in $K\{X\}_{\Omega}$.
The following proposition easily implies that we may use results on the $G$-invariants
$K\langle X\rangle^G$ in the free associative algebra $K\langle X\rangle$
to describe the $G$-invariants $K\{X\}_{\Omega}^G$ for an arbitrary
subgroup $G$ of $GL_d(K)$.

\begin{proposition}\label{relation with invariants of free associative algebras}
\text{\rm (i)} Let
\[
f(x_1,\ldots,x_d)=\sum f_T(x_1,\ldots,x_d)\in K\{X\}_{\Omega},
\]
where $f_T(x_1,\ldots,x_d)\in \nu_T(KX,\ldots,KX)$.
Then $f(x_1,\ldots,x_d)$ is $G$-invariant if and only if
$\pi_T(f_T(x_1,\ldots,x_d))\in K\langle X\rangle^G$ for all
$\Omega$-trees $T$;

\text{\rm (ii)} The Hilbert series of $K\langle X\rangle^G$,
$K\{x\}_{\Omega}$ and $K\{X\}_{\Omega}^G$ are related as follows. If
\[
H(K\langle X\rangle^G,t)=\sum_{m\geq 1}a_mt^m,\quad
H(\{x\}_{\Omega},t)=\sum_{m\geq 1}b_mt^m,
\]
then
\[
H(\{X\}_{\Omega}^G,t)=\sum_{m\geq 1}a_mb_mt^m.
\]
\end{proposition}

\begin{proof}
Since $GL_d(K)$ sends
$\nu_T(x_{j_1},\ldots,x_{j_N})$ to a linear combination of
monomials of the same kind, we obtain immediately that each $G$-invariant is a
linear combination of $G$-invariants $f_T\in \nu_T(KX,\ldots,KX)$,
which establishes (i). The proof of (ii) follows from the equality
\[
K\{X\}_{\Omega}^G=\bigoplus \nu_T(KX,\ldots,KX)^G,
\]
where the direct sum of vector spaces is on all  $\Omega$-trees $T$.
\end{proof}

Let $G$ be an arbitrary subgroup of $GL_d(K)$ and let us consider
the action of $G$ on $KX$. Since every basis of the vector space
$KX$ is a system of free generators of $K\{X\}_{\Omega}$, we fix a
basis of the subspace of $G$-invariants $(KX)^G$ and assume that
$\{x_1,\ldots,x_{d_0}\}\subset X$ is a basis of $(KX)^G$. Then we
complete this system to a basis of the whole $KX$ by
$x_{d_0+1},\ldots,x_d$. Obviously, every $\Omega$-polynomial
$f(x_1,\ldots,x_{d_0})$ is a $G$-invariant. We call such
polynomials {\it obvious invariants}.

\begin{example}\label{Weitzenboeck 2 x 2}
Let $d=4$, let $G$ be the cyclic subgroup of $GL_4(K)$ generated
by the matrix
\[
g=\left(
\begin{matrix}
1&0&1&0\\
0&1&0&1\\
0&0&1&0\\
0&0&0&1\\
\end{matrix}
\right),
\]
and let $\Omega=\underline{3}=\{\nu_3\}$ consist of one ternary operation
only. The subspace of $G$-invariants of $KX$ is
two-dimensional and is spanned by $x_1,x_2$. The polynomial
\[
f(x_1,x_2,x_3,x_4)=\nu_3(x_3,x_1,x_2)-\nu_3(x_1,x_1,x_4)
\]
is a nonobvious $G$-invariant because $g(f)=f$ and $f$ depends on
variables different from $x_1,x_2$.
\end{example}

The next result shows that the algebra of invariants is not
finitely generated in all nontrivial cases.

\begin{theorem}\label{nonfinitely generated invariants}
Let $G$ be any subgroup of $GL_d(K)$ and let $K\{X\}_{\Omega}^G$
contain nonobvious invariants. Then the algebra
$K\{X\}_{\Omega}^G$ is not finitely generated.
\end{theorem}

\begin{proof}
We choose a homogeneous nonobvious invariant of minimal degree. We
may assume that it is of the form
\begin{equation}\label{explicit form of f}
f=\sum_{j=1}^m\alpha_jw_j,\quad \alpha_j\in
K,w_j=\nu_T(x_{j_1},\ldots,x_{j_N})\in\{X\}_{\Omega},
\end{equation}
where all $w_j$ correspond to the same $\Omega$-tree $T$ and there
is no $w_j$ which depends only on the $G$-invariant variables
$x_1,\ldots,x_{d_0}$. Let, for some $k=1,\ldots,N$, and some
$w_j$, the $k$-th coordinate $x_{j_k}$ in
$w_j=\nu_T(x_{j_1},\ldots,x_{j_N})$ be a noninvariant variable,
i.e., $j_k>d_0$. We order the operations from $\Omega$ first by
degree and then in an arbitrary way. We fix the admissible
ordering on $\{X\}_{\Omega}$ which compares the $\Omega$-monomials
$w=\nu_{ni}(u_1,\ldots,u_n)$ first by total degree
$\text{deg}(w)$, then by the degree
$\text{deg}_{\text{noninv}}(w)$ with respect to the noninvariant
variables $x_{d_0+1},\ldots,x_d$, then by the first (outer)
operation $\nu_{ni}$, and then lexicographically. In the special
case of $w=\nu_T(x_{p_1},\ldots,x_{p_N})$,
where the $\Omega$-tree $T$ is as in (\ref{explicit form of f}), we start the
lexicographic ordering with the $k$-th position. Hence,
\[
w'=\nu_{n_1i_1}(u_1,\ldots,u_{n_1})\prec
\nu_{n_2i_2}(v_1,\ldots,v_{n_2})=w''
\]
means that

\noindent 1) $\text{deg}(w')<\text{deg}(w'')$;

\noindent 2) or $\text{deg}(w')=\text{deg}(w'')$,
$\text{deg}_{\text{noninv}}(w')<\text{deg}_{\text{noninv}}(w'')$;

\noindent 3) or $\text{deg}(w')=\text{deg}(w'')$,
$\text{deg}_{\text{noninv}}(w')=\text{deg}_{\text{noninv}}(w'')$,
$n_1<n_2$ or $n_1=n_2$, $i_1<i_2$;

\noindent 4) $\text{deg}(w')=\text{deg}(w'')$,
$\text{deg}_{\text{noninv}}(w')=\text{deg}_{\text{noninv}}(w'')$,
$\nu_{n_1i_1}=\nu_{n_2i_2}$ and $u_1=v_1,\ldots,u_{c-1}=v_{c-1}$,
$u_c\prec v_c$.

If in 4) both $w'$ and $w''$ are of the same type
$w'=\nu_T(x_{p_1},\ldots,x_{p_N})$ and
$w''=\nu_T(x_{q_1},\ldots,x_{q_N})$ for $T$ from (\ref{explicit form of f}),
we assume that first
$u_k\prec v_k$ and if $u_k=v_k$, then
$u_1=v_1,\ldots,u_{c-1}=v_{c-1}$, $u_c\prec v_c$. So, without loss
of generality we may assume that $k=1$, i.e., the first coordinate
$x_{j_1}$ of some $w_j$ is noninvariant.

We construct a sequence $f_1,f_2,\ldots$ of $G$-invariants,
starting with $f_1=f$. If in (\ref{explicit form of f}) each $w_j$
has the form
\[
w_j=\nu_T(x_{j_1},\ldots,x_{j_N})=\nu_{n1}(u_{jr_1},\ldots,u_{jr_n}),
\quad u_{jr_s}\in\{X\}_{\Omega},
\]
we define
\[
f_{k+1}=\sum_{j=1}^m\alpha_j\nu_{n1}(u_{jr_1},\ldots,u_{jr_{n-1}},
\nu_{n1}(u_{jr_n},\underbrace{f_k,\ldots,f_k}_{n-1})).
\]
In order to prove that $f_{k+1}$ is $G$-invariant, we use the
$GL_d(K)$-module isomorphism (\ref{module structure for fixed
tree}) which is also a $G$-module isomorphism and define the
$G$-module isomorphism
\[
\varphi_{k+1}:\nu_T(\underbrace{KX,\ldots,KX}_{n-1},\nu_{n1}(KX,\underbrace{f_k,\ldots,f_k}_{n-1}))
\to (KX)^{\otimes n}\otimes (Kf_k)^{\otimes (n-1)}
\]
by
\[
\varphi_{k+1}:\nu_T(x_{j_1},\ldots,x_{j_{N-1}},\nu_{n1}(x_{j_N},f_k,\ldots,f_k))
\to x_{j_1}\otimes\cdots\otimes x_{j_N}\otimes f_k^{\otimes
(n-1)}.
\]
Then
\[
\varphi_{k+1}(f_{k+1})=\sum_{j=1}^m\alpha_j\pi_T(f)\otimes
f_{k+1}^{\otimes (n-1)},
\]
which is $G$-invariant.

If the leading term of $f$ with respect to the introduced
admissible ordering is
\[
\overline{f}=\nu_{n1}(u_1^0,\ldots,u_{n-1}^0,u_n^0),\quad
u_1^0,\ldots,u_{n-1}^0,u_n^0\in\{X\}_{\Omega},
\]
then the $\Omega$-monomial $u_1^0$ depends also on a noninvariant
variable. Also, it is easy to see that the leading term of
$f_{k+1}$ is
\[
\nu_{n1}(u_1^0,\ldots,u_{n-1}^0,\nu_{n1}(u_n^0,\underbrace{\overline{f_k},\ldots,\overline{f_k}}_{n-1})).
\]
Now, let the algebra $K\{X\}_{\Omega}^G$ of $G$-invariants be
finitely generated by some $h_1,\ldots,h_m$. We may assume that the generators are homogeneous. We choose
a sufficiently large $k$ such that
$\text{deg}(f_{k+1})>\text{deg}(h_s)$, $s=1,\ldots,m$. Since
$f_{k+1}$ belongs to the $\Omega$-subalgebra of $K\{X\}_{\Omega}$
generated by $h_1,\ldots,h_m$, the leading term
$\overline{f_{k+1}}$ of $f_{k+1}$ can be expressed as an
$\Omega$-monomial of the leading terms $\overline{h_s}$ and is
different from them. Hence
\[
\overline{f_{k+1}}=\nu_{n1}(u_1^0,\ldots,u_{n-1}^0,
\nu_{n1}(u_n^0,\underbrace{\overline{f_k},\ldots,\overline{f_k}}_{n-1}))
=\nu_{n_1i_1}(v_1,\ldots,v_{n_1}),
\]
where each $v_p$ is the leading term of an element of the
$\Omega$-subalgebra generated by $h_1,\ldots,h_m$. This implies
that $\nu_{n1}=\nu_{n_1i_1}$ and
$u_1^0=v_1,\ldots,u_{n-1}^0=v_{n-1}$,
$\nu_{n1}(u_n^0,\overline{f_k},\ldots,\overline{f_k})=v_n$. Hence
$u_1^0$ is a leading term of a $G$-invariant element.
This is impossible because $\text{deg}(f)>\text{deg}(u_1^0)$ and
$\text{deg}(f)$ is the minimal degree of an invariant depending
not only on $G$-invariant variables.
\end{proof}

As an illustration of the proof, continuing Example
\ref{Weitzenboeck 2 x 2}, we can now start with

\[
f_1=f=f(x_1,x_2,x_3,x_4)=\nu_3(x_3,x_1,x_2)-\nu_3(x_1,x_1,x_4)
\]
and construct
\[
f_{k+1}=\nu_3(x_3,x_1,\nu_3(x_2,f_k,f_k))-\nu_3(x_1,x_1,\nu_3(x_4,f_k,f_k)).
\]

\begin{remark}
The main steps of the description of Koryukin \cite{Ko1} of the finitely generated algebras of invariants
$K\langle X\rangle^G$ can be applied also to the case of the free
binary algebra $K\{X\}$. Tracing his proof we obtain that if
$K\{X\}^G$ is finitely generated and $K\{X\}^G\subset K\{x_1,\ldots,x_{d_1}\}$, where $d_1$ is minimal with this
property (with respect to all linear changes of the system of generators
of $K\{X\}$), then $G$ acts on $Kx_1\oplus\cdots\oplus Kx_{d_1}$
as a finite cyclic group by scalar multiplications. If the order of this cyclic group is equal to $r$,
then the final arguments of our proof of Theorem \ref{nonfinitely generated invariants}
show that the elements $x_1^{r-1}(x_1x_1^{rk})$, $k=1,2,\ldots$,
do not belong to any finitely generated subalgebra of $K\{X\}^G$.
\end{remark}

\begin{corollary}
Let $\Omega\not=\emptyset$, i.e., $K\{X\}_{\Omega}$ is not the
vector space $KX$ with trivial multiplication and let
$G\not=\langle e\rangle$ be a finite subgroup of $GL_d(K)$. If the
algebra of invariants $K\{X\}_{\Omega}^G$ is nonzero, then it is
not finitely generated.
\end{corollary}

\begin{proof}
If $K\{X\}_{\Omega}^G$ contains a nonobvious invariant, then we
apply directly Theorem \ref{nonfinitely generated invariants}. Let
us assume that all invariants are obvious and depend on the $G$-invariant variables
$x_1,\ldots,x_{d_0}$. Since $K\{X\}_{\Omega}^G\not=0$ and $G\not=\langle e\rangle$,
we derive that $0<d_0<d$.
We use the well known
fact that the tensor powers of any faithful representation of a
finite group contain all irreducible representations of the group,
including the trivial representation.
Applying the theorem of Maschke, we choose the variables
$x_{d_0+1},\ldots,x_d$ to span a $G$-invariant complement of
$Kx_1\oplus \cdots\oplus Kx_{d_0}$. The representation of $G$ in
$\text{span}\{x_{d_0+1},\ldots,x_d\}$ is faithful. Hence there
exists a $G$-invariant $h\in
(\text{span}\{x_{d_0+1},\ldots,x_d\})^{\otimes p}$ of the form
\[
h(x_{d_0+1},\ldots,x_d)=\sum_{j=1}^m\alpha_jx_{j_1}\otimes\cdots\otimes
x_{j_p},\quad \alpha_j\in K, \quad j_1,\ldots,j_p>d_0.
\]
We choose an $\Omega$-tree $T$ with $N$ leaves, $N>p$, and
consider the associated operation $\nu_T$.
Now we consider the
isomorphism $\pi_T$ from (\ref{module structure for fixed tree}).
The variable $x_1$ is $G$-invariant and the $\Omega$-polynomial
\[
\pi_T^{-1}\left(\left(\sum_{j=1}^m\alpha_jx_{j_1}\otimes\cdots\otimes
x_{j_p}\right)\otimes x_1^{\otimes (N-p)}\right)
\]
\[
=\sum_{j=1}^m\alpha_j\nu_T(x_{j_1},\ldots,x_{j_p},\underbrace{x_1,\ldots,x_1}_{N-p})
\]
is a nonobvious $G$-invariant. This contradiction
completes the proof.
\end{proof}

Let us add in Example \ref{zero algebra of invariants} one more
variable $x_0$ which is fixed by $G$. Then $G$ is generated by the
$(d+1)\times (d+1)$ matrix
\[
g=\left(\begin{matrix}
1&0&\cdots&0\\
0&-1&\cdots&0\\
\vdots&\vdots&\ddots&\vdots\\
0&0&\cdots&-1\\
\end{matrix}\right),
\]
and $f=\nu_3(x_1,x_1,x_0)$ is a nonobvious $G$-invariant because
\[
g(f)=\nu_3(g(x_1),g(x_1),g(x_0))=\nu_3(-x_1,-x_1,x_0)=\nu_3(x_1,x_1,x_0)=g.
\]

The proof of Dicks and Formanek \cite{DF} that the finite cyclic groups
$G$ which act by scalar multiplication are the only groups such that
the algebra $K\langle X\rangle^G$ is finitely generated uses ideas very different
from the proof of Theorem \ref{nonfinitely generated invariants}
given above. Dicks and Formanek proved that for any finite group $G$ the Hilbert series of
$K\langle X\rangle^G$ satisfies
\begin{equation}\label{formula of Dicks and Formanek}
H(K\langle X\rangle^G,t)=\frac{1}{\vert G\vert}\sum_{g\in G}
\frac{1}{1-\text{tr}_{KX}(g)t},
\end{equation}
where $\text{tr}_{KX}(g)$ is the trace of the linear operator $g$ in
the $d$-dimensional vector space $KX$.
This is an analogue of the classical Molien formula
\begin{equation}\label{Molien formula}
H(K[X]^G,t)=\frac{1}{\vert G\vert}\sum_{g\in G}
\frac{1}{\text{det}(1-gt)}.
\end{equation}
Later, Formanek \cite{F1} generalized this result to the case
of any factor algebra $K\langle X\rangle/J$ of $K\langle X\rangle$
modulo a $GL_d(K)$-invariant ideal $J$.
The algebra $K\langle X\rangle/J$ inherits the multigrading
of $K\langle X\rangle$. Let $G$ be any finite subgroup of $GL_d(K)$
and let $\xi_1(g),\ldots,\xi_d(g)$ be the eigenvalues of the
$d\times d$ matrix $g\in G$. Then
\begin{equation}\label{formula of Formanek}
H((K\langle X\rangle/J)^G,t)=\frac{1}{\vert G\vert}\sum_{g\in G}
H(K\langle X\rangle/J,\xi_1(g)t,\ldots,\xi_d(g)t).
\end{equation}
Since the Hilbert series of $K[X]$ and $K\langle X\rangle$ are, respectively,
\[
H(K[X],t_1,\ldots,t_d)=\prod_{j=1}^d\frac{1}{1-t_j},
\]
\[
H(K\langle X\rangle,t_1,\ldots,t_d)=\frac{1}{1-(t_1+\cdots+t_d)},
\]
and since
\[
\prod_{j=1}^d(1-\xi_j(g)t)=\text{det}(1-gt),
\]
\[
\sum_{j=1}^d\xi_j(g)t=\text{tr}(gt),
\]
the formula (\ref{formula of Formanek}) gives immediately
(\ref{Molien formula}) and (\ref{formula of Dicks and Formanek}).
Applied to invariants of free $\Omega$-algebras, (\ref{formula of Formanek}) implies the following.

\begin{proposition}\label{Omega version of Formanek formula}
Let $G$ be a finite subgroup of $GL_d(K)$. Then
\begin{equation}\label{Hilbert series of invariants}
H(K\{X\}_{\Omega}^G,t)=\frac{1}{\vert G\vert}\sum_{g\in G}
H(K\{X\}_{\Omega},\xi_1(g)t,\ldots,\xi_d(g)t),
\end{equation}
where $\xi_1(g),\ldots,\xi_d(g)$ are the eigenvalues of the $d\times d$ matrix $g\in G$.
\end{proposition}

\begin{proof}
We recall the main steps of the proof of \cite{F1}, see also Section 6.3 of \cite{Dr3},
where the proof of (\ref{formula of Formanek}) is given as a sequence
of exercises. The first fact we need is that for a finite dimensional vector space $W$ and
a finite subgroup $G$ of $GL(W)$, the vector space of $G$-invariants in $W$,
i.e., $W^G=\{w\in W\mid g(w)=w,\quad g\in G\}$
coincides with the image of the Reynolds operator
$\phi:W\to W$ defined by
\[
\phi(w)=\frac{1}{\vert G\vert}\sum_{g\in G}g(w)
\]
and $\text{dim}(W^G)=\text{tr}_W(\phi)$, the trace of $\phi$ acting on $W$.
Further, if $g$ is a diagonalizable matrix acting on $KX$
(which is true when $g$ is of finite order), and $W$ is a $GL_d(K)$-invariant
multihomogeneous finite dimensional vector subspace of $K\langle X\rangle$,
then the Hilbert series $H(W,t_1,\ldots,t_d)$ of $W$ is a symmetric polynomial
in $t_1,\ldots,t_d$ and the trace of $g$ acting on $W$ is
\begin{equation}\label{trace of g acting on W}
\text{tr}_W(g)=H(W,\xi_1(g),\ldots,\xi_d(g)).
\end{equation}
In particular, for the homogeneous components $W^{(k)}$ of degree $k$
\begin{equation}\label{trace of g acting on homogeneous component}
\text{tr}_{W^{(k)}}(g)t^k=H(W^{(k)},\xi_1(g)t,\ldots,\xi_d(g)t).
\end{equation}
Finally, the equations
(\ref{trace of g acting on W}) and
(\ref{trace of g acting on homogeneous component})
imply for the graded subspace of the $G$-invariants in $W$
\[
H(W^G,t)=\sum_{k\geq 0}\text{dim}((W^G)^{(k)})t^k
=\sum_{k\geq 0}\text{tr}_{W^{(k)}}(\phi)t^k
\]
\[
=\frac{1}{\vert G\vert}
\sum_{g\in G}\sum_{k\geq 0}\text{tr}_{W^{(k)}}(g)t^k
=\frac{1}{\vert G\vert}\sum_{g\in G}H(W,\xi_1(g)t,\ldots,\xi_d(g)t).
\]
Since each homogeneous component $K\{X\}_{\Omega}^{(k)}$ of
$K\{X\}_{\Omega}$ is isomorphic as a $GL_d(K)$-module to a direct sum
of several copies of $K\langle X\rangle^{(k)}$, the equations
(\ref{trace of g acting on W}) and
(\ref{trace of g acting on homogeneous component})
hold also for the finite dimensional $GL_d(K)$-submodules $W$ of $K\{X\}_{\Omega}$.
This implies the formula (\ref{Hilbert series of invariants}).
\end{proof}

\begin{example}\label{binary trees with odd branches}
Let $G=\{e,-e\}$ act on the free nonassociative binary algebra $K\{x\}$,
where $-e$ changes the sign of $x$. As in Example \ref{zero algebra of invariants},
$K\{x\}^G$ is spanned by all homogeneous monomials of
even degree. Proposition \ref{Omega version of Formanek formula} gives that
\[
H(K\{x\}^G,t)=\frac{1}{2}\left(H(K\{x\},t)+H(K\{x\},-t)\right)
\]
(which can be seen also directly).
One can replace $H(K\{x\},t)$ with its explicit form
(\ref{generating function for Catalan numbers}),
as the generating function of the Catalan numbers,
and obtain
\[
H(K\{x\}^G,t)=\frac{1}{4}\left(2-\sqrt{1-4t}-\sqrt{1+4t}\right).
\]
Using the formula (\ref{generating function of free generators of subalgebras})
we derive for the generating function of any homogeneous set $Y$ of free
generators of $K\{x\}^G$
\begin{equation}\label{generators of even invariants}
G(Y,t)=\frac{1-\sqrt{1-16t^2}}{8}=\frac{1}{4}H(K\{x\},4t^2).
\end{equation}
Instead, we may proceed in the following way, with possible generalizations. Let
$H=H(K\{x\},t)=H_0+H_1$, where
\[
H_0=\frac{1}{2}(H(t)+H(-t)),\quad H_1=\frac{1}{2}(H(t)-H(-t))
\]
are, respectively, the even and the odd components of the series $H$.
Since
\[
0=H^2-H+t=(H_0^2+H_1^2-H_0)+(2H_0H_1-H_1+t),
\]
we separate the even and odd parts of this equation and obtain
\[
H_0^2+H_1^2-H_0=2H_0H_1-H_1+t=0,\quad H_1=\frac{t}{1-2H_0},
\]
\[
H_0^2+\frac{t^2}{(1-2H_0)^2}-H_0=0,\quad
(H_0^2-H_0)(1+4(H_0^2-H_0))+t^2=0.
\]
Using that $H_0^2-H_0+G(Y,t)=0$, we derive
\[
-G(Y,t)(1-4G(Y,t))+t^2=0,\quad (4G(Y,t))^2-(4G(Y,t))+4t^2=0.
\]
Hence, by (\ref{equation for u and v}) we obtain $4G(Y,t)=H(4t^2)$, i.e.,
(\ref{generators of even invariants}) and
\begin{equation}\label{formula for odd trees}
g_{2k}=\frac{1}{4^{k-1}}c_k,\quad k\geq 1,
\end{equation}
where $g_{2k}$ are the nonzero coefficients of the generating function $G(Y,t)$.

Since $K\{x\}^G$ is spanned by all nonassociative monomials of even degree,
the monomials of the form $uv$, where both $u,v$ are of odd degree form a
free generating set of $K\{x\}^G$. Applying the Stirling formula to $c_{2k}$ and $g_{2k}$
we obtain
\[
c_{2k}=\frac{1}{2k}\binom{4k-2}{2k-1}=\frac{(2(2k-1))!}{2k((2k-1)!)^2}
\approx \frac{4^{2k}}{8k\sqrt{\pi(2k-1)}},
\]
\[
g_{2k}=\frac{4^{k-1}(2(k-1))!}{k((k-1)!)^2}\approx \frac{4^{2k}}{16k\sqrt{\pi(k-1)}},
\]
\[
\frac{g_{2k}}{c_{2k}}\approx \frac{1}{2}\sqrt{\frac{2k-1}{k-1}},\quad
\lim_{k\to\infty}\frac{g_{2k}}{c_{2k}}=\frac{\sqrt{2}}{2}.
\]
Hence the quotient between the number of nonassociative monomials which are products of two monomials
of odd degree to the number of all monomials of the same degree tends to $\sqrt{2}/2$. For example,
\[
c_2=1,\quad g_2=1,\quad \frac{g_2}{c_2}=1,
\]
\[
c_4=5,\quad g_4=4,\quad \frac{g_4}{c_4}=0.8,
\]
\[
c_6=42,\quad g_6=32,\quad \frac{g_6}{c_6}\approx 0.761905,
\]
\[
c_8=429,\quad g_8=320,\quad \frac{g_8}{c_8}\approx 0.745921,
\]
\[
c_{10}=4862,\quad g_{10}=3584,\quad \frac{g_{10}}{c_{10}}\approx 0.737145,
\]
\[
c_{12}=58786,\quad g_{12}=43008,\quad \frac{g_{12}}{c_{12}}\approx 0.731603,
\]
\[
c_{14}=742900,\quad g_{14}=540672,\quad \frac{g_{14}}{c_{14}}\approx 0.727786,
\]
\[
c_{16}=9694845,\quad g_{16}=7028736,\quad \frac{g_{16}}{c_{16}}\approx 0.724997,
\]
\[
\frac{g_{20}}{c_{20}}\approx 0.721197,
\quad \frac{g_{30}}{c_{30}}\approx 0.716308,
\quad \frac{g_{40}}{c_{40}}\approx 0.713938,
\]
\[
\frac{g_{50}}{c_{50}}\approx 0.712539,
\quad \frac{g_{100}}{c_{100}}\approx 0.709790,
\]
which is very close to
\[
\frac{\sqrt{2}}{2}\approx 0.707105.
\]
The monomials which are products of two monomials of odd degree
can be labeled by  planar binary rooted trees with two branches
with an odd number of leaves for each branch. Hence there are much
more of such trees (about
\[
\frac{\sqrt{2}/2}{1-\sqrt{2}/2}\approx 2.41420
\]
times) than of  planar binary rooted trees with the same number of
leaves which have two branches with an even number of leaves.
\end{example}

In the case of infinite groups $G$ the Molien formula (\ref{Molien formula}) has no formal sense.
Nevertheless, if $G$ is compact, one can define
Haar measure on $G$, replace the sum with an integral
and obtain the Molien-Weyl formula for the Hilbert series of the algebra of
invariants, see \cite{We1, We2}. The analogue of (\ref{formula of Dicks and Formanek})
for $G$ infinite is given by Almkvist, Dicks and Formanek \cite{ADF}.
For other applications of the Molien-Weyl formula for objects related with noncommutative algebra
see also \cite{F1}, \cite{Dr2} and the books \cite{F2}, \cite{DrF}.
We shall consider such applications in the next section.

Without presenting a comprehensive survey, we shall mention several results
of action of other objects, different from groups, in the spirit of invariant theory.
Instead of invariants of linear groups one may consider constants of derivations.
A theorem of Jooste \cite{J} and Kharchenko \cite{Kh1} gives that
for a Lie algebra $D$ of linear derivations
the algebra of constants
\[
K\langle X\rangle^D=\{f\in K\langle X\rangle\mid \delta(f)=0,\delta\in D\}
\]
is free again. See also Koryukin \cite{Ko2} and Ferreira and
Murakami \cite{FM1} for the problem of finite generation of
$K\langle X\rangle^D$. One may study also invariants of Hopf
algebras acting on $K\langle X\rangle$. We shall mention only
\cite{FMP} and \cite{FM2} which show that the algebra of
invariants of $K\langle X\rangle$ under the linear action of a
Hopf algebra $H$ is free and, under some mild restrictions on $H$,
the algebra $K\langle X\rangle^H$ is finitely generated only if
the action of $H$ is scalar. It is a natural problem to study
algebras of constants of derivations and invariants of Hopf
algebras also in the case of free $\Omega$-algebras.

\section{Weitzenb\"ock derivations, special linear groups and elliptic integrals}

As in the previous section, we assume that $K$ is a field of characteristic 0 and
$X=\{x_1,\ldots,x_d\}$. Almkvist, Dicks and Formanek \cite{ADF}
studied invariants of the special linear group $SL_p(K)$ and the
unitriangular group $UT_p(K)$ acting on the free associative algebra
$K\langle X\rangle$. They expressed the Hilbert series of the algebra
of invariants in terms of multiple integrals. We shall transfer these results
to the case of $K\{X\}_{\Omega}$. We shall consider in detail the case $p=2$ only.
Also, instead for the unitriangular group $UT_2(K)$ we shall state
the results for Weitzenb\"ock derivations and unipotent actions of the infinite cyclic group.

For details on representation theory of general linear groups see
e.g. the books by Macdonald \cite{Mc} and Weyl \cite{We2}. Let
$W(\lambda)=W(\lambda_1,\lambda_2)$ be the irreducible
$GL_2(K)$-module corresponding to the partition
$\lambda=(\lambda_1,\lambda_2)$. The role of the character of
$W(\lambda)$ is played by the Schur function
$s_{\lambda}(u_1,u_2)$. This means that if the matrix $g\in
GL_2(K)$ has eigenvalues $\xi_1,\xi_2$, then $g$ acts as a linear
operator on $W(\lambda)$ with trace
$\text{tr}_{W(\lambda)}(g)=s_{\lambda}(\xi_1,\xi_2)$. In the case
of two variables $s_{\lambda}(u_1,u_2)$ has the simple form
\[
s_{\lambda}(u_1,u_2)=u_1^{\lambda_1}u_2^{\lambda_2}+u_1^{\lambda_1-1}u_2^{\lambda_2+1}+\cdots
+u_1^{\lambda_2+1}u_2^{\lambda_1-1}+u_1^{\lambda_2}u_2^{\lambda_1}
\]
\[
=(u_1u_2)^{\lambda_2}\frac{u_1^{\lambda_1-\lambda_2+1}-u_2^{\lambda_1-\lambda_2+1}}{u_1-u_2}.
\]
The Schur functions form a basis of the vector space of symmetric polynomials in $K[u_1,u_2]$.

Recall that a linear operator $\delta$ of the $K$-algebra $R$ is called a derivation if
$\delta(uv)=\delta(u)v+u\delta(v)$ for all $u,v\in R$. Similarly,
the linear operator $\delta$ is a derivation of the $\Omega$-algebra $R$
if
\begin{equation}\label{definition of derivation}
\delta(\nu_{ni}(v_1,\ldots,v_n))
=\sum_{j=1}^n\nu_{ni}(v_1,\ldots,\delta(v_j),\ldots,v_n)
\end{equation}
for all $v_j\in R$ and all $\nu_{ni}\in \Omega$. The derivation is
locally nilpotent if for any $v\in R$ there is a $k$ such that
$\delta^k(v)=0$. If $\delta$ is a locally nilpotent derivation,
then the exponential of $\delta$
\[
\exp(\delta)=1+ \frac{\delta}{1!}+\frac{\delta^2}{2!}+\cdots
\]
is well defined and is an automorphism of $R$. The kernel
$\ker(\delta)=R^{\delta}$ of $\delta$
is called the algebra of constants of $\delta$. It coincides with the algebra
$R^{\exp(\delta)}$ of the fixed points of the automorphism
$\exp(\delta)$. Every mapping
$\delta:X\to K\{X\}_{\Omega}$ can be extended to a derivation of $K\{X\}_{\Omega}$:
If $v_1,\ldots,v_n$ are monomials in $\{X\}_{\Omega}$, then we define $\delta(\nu_{ni}(v_1,\ldots,v_n))$
inductively by (\ref{definition of derivation}).

Let $g\in GL_d(K)$ be a
unipotent linear operator acting on the vector space
$KX=Kx_1\oplus\cdots\oplus Kx_d$. By the
classical theorem of Weitzenb\"ock \cite{W}, the (commutative and associative) algebra of
invariants
\[
K[X]^g=\{f\in K[X]\mid
f(g(x_1),\ldots,g(x_d))=f(x_1,\ldots,x_d)\}
\]
is finitely generated. All eigenvalues of $g$ are equal to 1 and, up to a change
of the basis, $g$ is determined uniquely by its Jordan normal
form. Hence, for each fixed $d$ we may consider only a finite number of
linear unipotent operators.
Equivalently, we may consider the
linear locally nilpotent derivation $\delta=\log g$ called a
Weitzenb\"ock derivation. Clearly, the algebra of invariants $K[X]^g$
coincides with the algebra of constants $K[X]^{\delta}\ (=\ker(\delta)$).
See the book of Nowicki \cite{N} for
concrete generators of $K[X]^{\delta}$ for small $d$
and the book by Freudenburg \cite{Fr} for general information on
locally nilpotent derivations of polynomial algebras. The paper by Drensky
and Gupta \cite{DrG} deals with Weitzenb\"ock derivations for free
and relatively free associative and Lie algebras.
We present a short account on the properties of the algebra of
constants $K\{X\}_{\Omega}^{\delta}$ and show how to calculate the Hilbert
series of $K\{X\}_{\Omega}^{\delta}$.
The proofs are based on the description of the invariants of the
group of unitriangular matrices given by De Concini, Eisenbud and
Procesi \cite{DEP} and the work of Almkvist, Dicks and Formanek
\cite{ADF}. We assume that $\delta$ is a linear locally nilpotent
derivation. It acts on $KX$ as a nilpotent linear operator, with Jordan normal form
consisting of $k$ cells of size $n_1+1,\ldots,n_k+1$, respectively, and $X$
is a Jordan basis of $\delta$. Hence either $\delta(x_j)=x_{j-1}$ or
$\delta(x_j)=0$, $j=1,\ldots,d$.

We equip $KX$ with a $GL_2(K)$-module structure in the following
way. If $X_r=\{x_{j_0},x_{j_0+1},\ldots,x_{j_0+n_r}\}$ is the part
of the basis $X$ corresponding to the $r$-th Jordan cell of
$\delta$, we assume that $GL_2(K)$ acts on $KX_r$ as on the
$GL_2(K)$-module of commutative and associative polynomials,
homogeneous of degree $n_r$, in two variables $x,y$: If $g\in
GL_2(K)$ and
\[
g(x^{n_r-m}y^m)=\sum_{q=0}^{n_r}\alpha_{qm}x^{n_r-q}y^q
\]
for some $\alpha_{qm}\in K$, then
\[
g(x_{j_0+m})=\sum_{q=0}^{n_r}\alpha_{qm}x_{j_0+q}.
\]
Hence $KX$ is isomorphic to the direct sum
$W(n_1,0)\oplus\cdots\oplus W(n_k,0)$ as $GL_2(K)$-module.
 We extend the
action of $GL_2(K)$ diagonally to $K\{X\}_{\Omega}$. Then, see
\cite{DEP} and \cite{ADF}, each irreducible $GL_2(K)$-sub\-module
$W(\lambda_1,\lambda_2)$ of $K\{X\}_{\Omega}$ contains a
one-dimensional $\delta$-constant subspace and the algebra
$K\{X\}_{\Omega}^{\delta}$ is spanned by these subspaces. We
define on $K\{X\}_{\Omega}$ a ${\mathbb Z}^3$-grading assuming
that the degree of $x_{j_0+m}$ from $X_r$ is equal to
$(n_r-m,m,1)$ and consider the Hilbert series
\begin{equation}\label{Hilbert series for Weitzenboeck action}
H_{\delta}(K\{X\}_{\Omega},u_1,u_2,z)
=H(K\{x\}_{\Omega},z\sum_{q=1}^ks_{(n_r,0)}(u_1,u_2)).
\end{equation}
It is obtained from the Hilbert series
$H(K\{X\}_{\Omega},t_1,\ldots,t_d)$ by replacing the variables
$t_1,t_2,\ldots,t_d$ respectively by
\[
u_1^{n_1}z,u_1^{n_1-1}u_2z,
\ldots,u_1u_2^{n_1-1}z,u_2^{n_1}z,\ldots,
\]
\[
u_1^{n_k}z,u_1^{n_k-1}u_2z,
\ldots,u_1u_2^{n_k-1}z,u_2^{n_k}z.
\]
The variables $u_1,u_2$ take into account the bigrading related to
the $GL_2(K)$-action and the extra variable $z$ counts the usual
grading.

The function $H_{\delta}(K\{X\}_{\Omega},u_1,u_2,z)$
is symmetric in $u_1,u_2$. The coefficient of its linear in $z$
component is equal to
\[
\sum_{i=1}^ks_{(n_i,0)}(u_1,u_2)
\]
which is the character of the $GL_2(K)$-module $KX$.

Hence $H_{\delta}(K\{X\}_{\Omega},u_1,u_2,z)$ is
the character of the $GL_2(K)$-module $K\{X\}_{\Omega}$.
By \cite{ADF}, this means that if
\[
H_{\delta}(K\{X\}_{\Omega},u_1,u_2,z)=
\sum_{q\geq 1}\left(\sum_{\lambda\vdash q}
m_{\lambda}s_{\lambda}(u_1,u_2)\right)z^q,
\]
then the homogeneous component of degree $q$ decomposes as
\begin{equation}\label{homogeneous component of free algebra}
K\{X\}_{\Omega}^{(q)}=\bigoplus_{\lambda\vdash q} m_{\lambda}W(\lambda_1,\lambda_2).
\end{equation}

\begin{theorem}\label{Weitzenboeck}
Let $\delta$ be a linear locally nilpotent derivation of
$K\{X\}_{\Omega}$, $X=\{x_1,\ldots,x_d\}$, which, when acting on $KX$,
has a Jordan normal form consisting
of $k$ cells of size $n_1+1,\ldots,n_k+1$, respectively.
Then the Hilbert series of the algebra of constants
$K\{X\}_{\Omega}^{\delta}$ is given by
\begin{equation}\label{Hilbert series of Weitzenboeck constants}
H(K\{X\}_{\Omega}^{\delta},z)
=2\int_0^1\cos^2(\pi u)H_{\delta}(K\{X\}_{\Omega},e^{2\pi iu},e^{-2\pi iu},z)du,
\end{equation}
where $H_{\delta}(K\{X\}_{\Omega},u_1,u_2,z)$ is defined in
\text{\rm (\ref{Hilbert series for Weitzenboeck action})}. Equivalently,
\begin{equation}\label{compact form of Hilbert series of Weitzenboeck constants}
H(K\{X\}_{\Omega}^{\delta},z)
=2\int_0^1\cos^2(\pi u)H\left(K\{x\}_{\Omega},
z\sum_{i=1}^k\frac{\sin(2\pi (n_i+1)u)}{\sin(2\pi u)}\right)du.
\end{equation}
\end{theorem}

\begin{proof} We follow the proof of Almkvist, Dicks, and Formanek
\cite{ADF} for the Hilbert series of $K\langle X\rangle^{\delta}$.
If $K\{X\}_{\Omega}^{(q)}$ decomposes as in
(\ref{homogeneous component of free algebra}),
then the Hilbert series of the algebra of $\delta$-constants is
\[
H(K\{X\}_{\Omega}^{\delta},z)
=\sum_{q\geq 1}\left(\sum_{\lambda\vdash q} m_{\lambda}\right)z^q.
\]
Hence, for the proof of (\ref{Hilbert series of Weitzenboeck constants})
it is sufficient to show that
\[
2\int_0^1\cos^2(\pi u)s_{\lambda}(e^{2\pi iu},e^{-2\pi iu})du=1,
\]
which was already used in the proof of \cite{ADF} (or may be verified
directly). The expression (\ref{compact form of Hilbert series of Weitzenboeck constants})
follows from the formula
\[
s_{(n,0)}(e^{2\pi iu},e^{-2\pi iu})
=\frac{e^{2\pi i(n+1)u}-e^{-2\pi i(n+1)u}}{e^{2\pi iu}-e^{-2\pi iu}}
=\frac{\sin(2\pi (n+1)u)}{\sin(2\pi u)}.
\]
\end{proof}

\begin{example}\label{elliptic integrals for free binary algebra}
Let $K\{X\}_{\Omega}=K\{X\}$ be the free nonassociative algebra,
i.e., $\Omega$ consist of one binary operation only.

(i) Let $d=2$ and $\delta(x_1)=0$, $\delta(x_2)=x_1$, i.e.,
\begin{equation}\label{delta in dimension 2}
\delta=\left(\begin{matrix}
0&1\\
0&0\\
\end{matrix}\right).
\end{equation}
By \cite{DrG} the Hilbert series of
$K\langle X\rangle^{\delta}$ (for the nonunitary algebra
$K\langle X\rangle$) is
\[
H(K\langle X\rangle^{\delta},t)
=\sum_{p\geq 0}\binom{2p+1}{p}t^{2p+1}
+\sum_{p\geq 1}\binom{2p}{p}t^{2p}.
\]
The algebra $K\langle X\rangle^{\delta}$ is a free associative algebra
and has a homogeneous set of free generators $Y$ with generating function
\begin{equation}\label{free generators of Weitzenboeck associative algebras}
G(Y,t)=t+\sum_{p\geq 1}c_{p+1}t^{2p}.
\end{equation}
Hence Proposition \ref{relation with invariants of free associative algebras} gives
\[
H(K\langle X\rangle^{\delta},t)
=\sum_{p\geq 0}\binom{2p+1}{p}c_{2p+1}t^{2p+1}
+\sum_{p\geq 1}\binom{2p}{p}c_{2p}t^{2p},
\]
where $c_n$ are the Catalan numbers.

Applying Theorem \ref{Weitzenboeck} we obtain
\[
H(K\{x\},t)=\frac{1-\sqrt{1-4t}}{2},
\]
\[
H_{\delta}(K\{X\},u_1,u_2,z)=\frac{1-\sqrt{1-4(u_1+u_2)z}}{2},
\]
\[
H(K\{X\}^{\delta},z)
=\int_0^1\cos^2(\pi u)\left(1-\sqrt{1-8z\cos(2\pi u)}\right)du,
\]
which is an elliptic integral. If $Y$ is a homogeneous set
of free generators of $K\{X\}^{\delta}$, then its generating
function is
\[
G(Y,z)=H(K\{X\}^{\delta},z)-H(K\{X\}^{\delta},z)^2.
\]
The beginning of the expansion of $G(Y,z)$ as a formal power series is
\[
G(Y,z)=z+z^2+2z^3+14z^4+56z^5+404z^6+2020z^7+\cdots .
\]
(Compare with the expansion of the generating function
(\ref{free generators of Weitzenboeck associative algebras}) of the free
generators of $K\langle X\rangle^{\delta}$.)
For the free generators of degree $\leq 3$ of $K\{X\}^{\delta}$ one may choose
\[
x_1,\quad x_1x_2-x_2x_1,\quad
(x_1x_1)x_2-(x_2x_1)x_1,\quad x_1(x_1x_2)-x_2(x_1x_1).
\]

(ii) Let $d=3$ and
\begin{equation}\label{delta in dimension 3}
\delta=\left(\begin{matrix}
0&1&0\\
0&0&1\\
0&0&0\\
\end{matrix}\right).
\end{equation}
Then (\ref{Hilbert series of Weitzenboeck constants}) and
(\ref{compact form of Hilbert series of Weitzenboeck constants}) give
\[
H(K\{X\}^{\delta},z)
=\int_0^1\cos^2(\pi u)\left(1-\sqrt{1-4z(1+2\cos(4\pi u))}\right)du.
\]
The beginning of the generating expansion for the free generators is
\[
z+2z^2+8z^3+58z^4+440z^5+3728z^6+33088z^7+\cdots .
\]
\end{example}

\begin{example}\label{elliptic integrals for free omega-algebra}
Let $K\{X\}_{\Omega}=K\{X\}_{\omega}$, i.e.,
$\Omega_n$ consists of one operation for each $n\geq 2$.
Applying Theorem \ref{Weitzenboeck} to (\ref{equation for omega}) in
Example \ref{examples of Hilbert series} (ii)
we obtain for $d=2$ and $\delta$ from (\ref{delta in dimension 2})
\[
H(K\{X\}_{\omega}^{\delta},z)=\frac{1}{2}\int_0^1\cos^2(\pi u)
f(2z\cos(2\pi u))du,
\]
where
\[
f(t)=1+t-\sqrt{1-6t+t^2}.
\]
For $d=3$ and $\delta$ from (\ref{delta in dimension 3}), we
obtain again elliptic integrals:
\[
H(K\{X\}_{\omega}^{\delta},z)=\frac{1}{2}\int_0^1\cos^2(\pi u)
f(z(1+2\cos(4\pi u)))du,
\]
which is, in low degrees,
\[
z+3z^2+21z^3+209z^4+2295z^5+27777z^6+354879z^7+\cdots
\]
\end{example}

\bigskip

Following Almkvist, Dicks and Formanek \cite{ADF}, we consider
a polynomial representation of $GL_2(K)$ in $KX$, i.e., we assume
that $KX$ has the $GL_2(K)$-module structure
\begin{equation}\label{GL-module KX}
KX\cong W(\lambda_1^{(1)},\lambda_2^{(1)})\oplus\cdots\oplus
W(\lambda_1^{(k)},\lambda_2^{(k)}).
\end{equation}
This induces also a representation of $SL_2(K)\subset GL_2(K)$.
We translate the results of \cite{ADF}
for the Hilbert series of $K\langle X\rangle^{SL_2(K)}$
to the case of $K\{X\}_{\Omega}^{SL_2(K)}$.

\begin{theorem}\label{SL2}
Let the $GL_2(K)$-module structure of $KX$, $X=\{x_1,\ldots,x_d\}$,
be given by \text{\rm (\ref{GL-module KX})}. Then the Hilbert series
of the algebra of $SL_2(K)$-invariants in $K\{X\}_{\Omega}$ is
\[
H(K\{X\}_{\Omega}^{SL_2(K)},z)=2\int_0^1 \sin^2(2\pi
u)H\left(K\{x\}_{\Omega},z
\sum_{i=1}^ks_{(\lambda_1^{(i)},\lambda_2^{(i)})} (e^{2\pi
iu},e^{-2\pi iu})\right)du
\]
\[
=2\int_0^1\sin^2(2\pi u)H\left(K\{x\}_{\Omega},z\sum_{i=1}^k
\frac{\sin(2\pi (\lambda_1^{(i)}-\lambda_2^{(i)}+1)u)}{\sin(2\pi u)}\right)du.
\]
\end{theorem}

\begin{proof}
As in the proof of Theorem \ref{Weitzenboeck}, it is sufficient to
take into account that $W(\lambda_1,\lambda_2)$ contains
a one-dimensional $SL_2(K)$-invariant if $\lambda_1=\lambda_2$
and does not contain $SL_2(K)$-invariants otherwise. Then we need to show that
\[
2\int_0^1\sin^2(2\pi u)s_{\lambda}(e^{2\pi iu},e^{-2\pi iu})du
=\delta_{\lambda_1\lambda_2},
\]
where $\delta_{pq}$ is the Kronecker delta. This can be checked directly and
was already used in \cite{ADF}.
\end{proof}

\begin{example}
If $d=2$ and $GL_2(K)$ acts naturally on $KX$, i.e. $KX\cong W(1,0)$, then
the results in \cite{DrG} for the Hilbert series
$H(K\langle X\rangle^{SL_2(K)},t)$ of the $SL_2(K)$-invariants
of the nonunitary free associative algebra give
\[
H(K\langle X\rangle^{SL_2(K)},t)
=\sum_{p\geq 1}c_{p+1}t^{2p}=\frac{1-2t^2-\sqrt{1-4t^2}}{2t^2}.
\]
By Proposition \ref{relation with invariants of free associative algebras}
(ii) we obtain that
\[
H(\{X\}^{SL_2(K)},z)=\sum_{p\geq 1}c_{2p}c_{p+1}z^{2p}.
\]
On the other hand, Theorem \ref{SL2} gives
\[
H(K\{X\}^{SL_2(K)},z)=\int_0^1
\sin^2(2\pi u)\left(1-\sqrt{1-8z\sin(2\pi u)}\right)du
\]
which is again an elliptic integral.
\end{example}

\begin{remark}
In order to obtain $\Omega$-analogues of the results in \cite{ADF}
for the invariants of $SL_r(K)$ and $UT_d(K)$ we equip $KX$
with the structure of a $GL_r(K)$-module with character (the trace
of $g\in GL_r(K)$ acting on $KX$)
\[
T_{KX}(u_1,\ldots,u_r)=\sum_{i=1}^k s_{\lambda^{(i)}}(u_1,\ldots,u_r).
\]
Then we replace in \cite{ADF} the $GL_r$-character
\[
H(K\langle X\rangle,u_1,\ldots,u_r)
=\frac{1}{1-zT_{KX}(u_1,\ldots,u_r)}
\]
with $H(K\{x\}_{\Omega},z\sum_{i=1}^kzT_{KX}(u_1,\ldots,u_r))$
and obtain integral expressions for
$H(K\{X\}_{\Omega}^{SL_r(K)},z)$ and $H(K\{X\}_{\Omega}^{UT_r(K)},z)$.
\end{remark}

\section*{Acknowledgements}

The first author wants to thank for the hospitality at the
Department of Mathematics of the Ruhr University in Bochum during
his visit there. Both authors thank Lothar Gerritzen for
stimulating discussions about the subject. They are especially indebted
to the anonymous referee for the historical comments which
led to the improvement of the exposition and the extension of the list of references.

\end{document}